\PassOptionsToPackage{lowtilde}{url}       % to typeset URLs, see http://www.latex-community.org/viewtopic.php?f=44&t=3663
\documentclass[a4paper,reqno]{amsart}
\usepackage{Local_Definitions}
\usepackage{overpic}
\usepackage{tikz}

\begin{document}
\title[Discrete TV of the Normal as Shape Prior]{Discrete Total Variation of the Normal Vector Field as Shape Prior with Applications in Geometric Inverse Problems}
\date{August 19, 2019}
\author[R.~Bergmann]{Ronny Bergmann}
\address{Technische Universität Chemnitz, Faculty of Mathematics, 09107 Chemnitz, Germany}
\email{ronny.bergmann@mathematik.tu-chemnitz.de}
\urladdr{https://www.tu-chemnitz.de/mathematik/part\_dgl/people/bergmann}
%\thanks{thanks}

\author[M.~Herrmann]{Marc Herrmann}
\address{Julius-Maximilians-Universität Würzburg, Faculty of Mathematics and Computer Science, Lehrstuhl für Mathematik~VI, Emil-Fischer-Straße~40, 97074 Würzburg, Germany}
\email{marc.herrmann@mathematik.uni-wuerzburg.de}
\urladdr{https://www.mathematik.uni-wuerzburg.de/~herrmann}
%\thanks{thanks}

\author[R.~Herzog]{Roland Herzog}
\address{Technische Universität Chemnitz, Faculty of Mathematics, 09107 Chemnitz, Germany}
\email{roland.herzog@mathematik.tu-chemnitz.de}
\urladdr{https://www.tu-chemnitz.de/herzog}
%\thanks{thanks}

\author[S.~Schmidt]{Stephan Schmidt}
\address{Julius-Maximilians-Universität Würzburg, Faculty of Mathematics and Computer Science, Lehrstuhl für Mathematik~VI, Emil-Fischer-Straße~40, 97074 Würzburg, Germany}
\email{stephan.schmidt@mathematik.uni-wuerzburg.de}
\urladdr{https://www.mathematik.uni-wuerzburg.de/~schmidt}
%\thanks{thanks}

\author[J.~Vidal-N{\'u}{\~n}ez]{Jos{\'e} Vidal-N{\'u}{\~n}ez}
\address{Technische Universität Chemnitz, Faculty of Mathematics, 09107 Chemnitz, Germany}
\email{jose.vidal-nunez@mathematik.tu-chemnitz.de}
\urladdr{https://www.tu-chemnitz.de/mathematik/part\_dgl/people/vidal}
%\thanks{thanks}

\begin{abstract}
	An analogue of the total variation prior for the normal vector field along the boundary of piecewise flat shapes in 3D is introduced.
	A major class of examples are triangulated surfaces as they occur for instance in finite element computations.
	The analysis of the functional is based on a differential geometric setting in which the unit normal vector is viewed as an element of the two-dimensional sphere manifold.
	It is found to agree with the discrete total mean curvature known in discrete differential geometry.
	A split Bregman iteration is proposed for the solution of discretized shape optimization problems, in which the total variation of the normal appears as a regularizer.
	Unlike most other priors, such as surface area, the new functional allows for piecewise flat shapes.
	As two applications, a mesh denoising and a geometric inverse problem of inclusion detection type involving a partial differential equation are considered.
	Numerical experiments confirm that polyhedral shapes can be identified quite accurately.
\end{abstract}

\keywords{total variation of the normal; discrete differential geometry; split Bregman iteration; shape optimization; geometric inverse problem; inclusion detection}

\maketitle

%------------------------------------------------------------------
\section{Introduction}
\label{sec:Introduction}
%------------------------------------------------------------------

The total variation (TV) functional is popular as a regularizer in imaging and inverse problems; see for instance \cite{RudinOsherFatemi1992,ChanGolubMulet1999,BachmayrBurger2009,Langer2017} and \cite[Chapter~8]{Vogel2002}.
It is most commonly applied to functions with values in $\R$ or $\R^n$.
In the companion paper \cite{BergmannHerrmannHerzogSchmidtVidalNunez2019:1_preprint}, we introduced the total variation of the normal vector field~$\bn$ along smooth surfaces $\Gamma \subset \R^3$:

\begin{equation}
	\label{eq:TV_of_normal}
	\abs{\bn}_{TV(\Gamma)}
	\coloneqq
	\int_\Gamma \bigh(){\absRiemannian{(D_\Gamma \bn) \, \bxi_1}^2 + \absRiemannian{(D_\Gamma \bn) \, \bxi_2}^2}^{1/2} \, \ds
	.
\end{equation}
In contrast to the setting of real- or vector-valued functions, the normal vector field is manifold-valued with values in the sphere $\sphere{2} = \{ \bv \in \R^3: \abs{\bv}_2 = 1 \}$.
In \eqref{eq:TV_of_normal}, $D_\Gamma \bn$ denotes the derivative (push-forward) of $\bn$, and $\{\bxi_1(\bs),\bxi_2(\bs)\}$ is an arbitrary orthonormal basis (w.r.t.\ the Euclidean inner product in the embedding $\Gamma \subset \R^3$) of the tangent spaces $\tangent{\bs}{\Gamma}$ along $\Gamma$. 
Finally, $\absRiemannian{\,\cdot\,}$ denotes the norm induced by a Riemannian metric on $\sphere{2}$.
It was shown in \cite{BergmannHerrmannHerzogSchmidtVidalNunez2019:1_preprint} that \eqref{eq:TV_of_normal} can be alternatively expressed as 
\begin{equation*}
	\abs{\bn}_{TV(\Gamma)}
	=
	\int_\Gamma \bigh(){k_1^2 + k_2^2}^{1/2} \, \ds,
\end{equation*}
where $k_1$ and $k_2$ are the principal curvatures of the surface.

In this paper, we discuss a discrete variant of \eqref{eq:TV_of_normal} tailored to piecewise flat surfaces $\Gamma_h$, where \eqref{eq:TV_of_normal} does not apply.
In contrast with the smooth setting, the total variation of the piecewise constant normal vector field $\bn$ is concentrated in jumps across edges between flat facets.
We therefore propose the following \emph{discrete total variation of the normal},
\begin{equation}
	\label{eq:tv_of_normal_discrete}
	\abs{\bn}_{DTV(\Gamma_h)}
	\coloneqq
	\sum_E d(\bn_E^+,\bn_E^-) \abs{E}.
\end{equation}
Here $E$ denotes an edge of length $\abs{E}$ between facets, and $d(\bn_E^+,\bn_E^-)$ is the geodesic distance between the two neighboring normal vectors.

We investigate \eqref{eq:tv_of_normal_discrete} in \Cref{sec:discrete_tv_of_normal}.
It turns out to coincide with the \emph{discrete total mean curvature} known in discrete differential geometry.
Subsequently, we discuss the utility of this functional as a prior in shape optimization problems cast in the form
\begin{equation}
	\label{eq:geometric_inverse_problem_with_TV_of_normal}
	\begin{aligned}
		& \text{Minimize} \quad \ell(u(\Omega_h),\Omega_h) + \beta \, \abs{\bn}_{DTV(\Gamma_h)} \\
		& \text{w.r.t.\ the vertex positions of the discrete shape $\Omega_h$ with boundary } \Gamma_h.
	\end{aligned}
\end{equation}
Here $u(\Omega_h)$ denotes the solution of the problem specific partial differential equation (PDE), which depends on the unknown domain $\Omega_h$.
Moreover, $\ell$ represents a loss function, such as a least squares function.
In particular, \eqref{eq:geometric_inverse_problem_with_TV_of_normal} includes geometric inverse problems, where one seeks to recover a \emph{shape} $\Omega_h \subset \R^3$ representing, e.g., the location of a source or inclusion inside a given, larger domain, or the geometry of an inclusion or a scatterer. 
Numerical experiments confirm that $\abs{\bn}_{DTV(\Gamma_h)}$, as a shape prior, can help to identify polyhedral shapes.

Similarly as for the case of smooth surfaces discussed in \cite{BergmannHerrmannHerzogSchmidtVidalNunez2019:1_preprint}, solving discrete shape optimization problems \eqref{eq:geometric_inverse_problem_with_TV_of_normal} is challenging due to the non-trivial dependency of $\bn$ on the vertex positions of the discrete surface~$\Gamma_h$, as well as the non-smoothness of $\abs{\bn}_{DTV(\Gamma_h)}$.
We therefore propose in \Cref{sec:discrete_split_Bregman} a version of the split Bregman method proposed in \cite{GoldsteinOsher2009}, an algorithm from the alternating direction method of multipliers (ADMM) class in which the jumps in the normal vector are treated as a separate variable.
The particularity here is that the normal vector has values in $\sphere{2}$ and thus the jump, termed $\bd$, is represented by a logarithmic map in the appropriate tangent space.
An outstanding feature of the proposed splitting is that the two subproblems, the minimization w.r.t.\ the vertex coordinates representing the discrete surface and w.r.t.\ $\bd$, are directly amenable to numerical algorithms.

Although many optimization algorithms have been recently generalized to Riemannian manifolds, see, e.g., \cite{Bacak2014,BergmannPerschSteidl2016,BergmannHerzogTenbrinckVidal-Nunez2019_preprint}, the Riemannian split Bregman method for manifolds proposed in this and the companion paper \cite{BergmannHerrmannHerzogSchmidtVidalNunez2019:1_preprint} is new to the best of our knowledge.
Its detailed investigation will be postponed to future work.
For a general overview of optimization on manifolds, we refer the reader to \cite{AbsilMahonySepulchre2008}.
We anticipate that our method can be applied to other non-smooth problems involving manifold-valued total variation functionals as well. 
Examples falling into this class have been introduced for instance in \cite{LellmannStrekalovskiyKoetterCremers2013,BergmannTenbrinck2018}.
An alternative splitting scheme, the so-called half-quadratic minimization, was introduced by~\cite{BergmannChanHielscherPerschSteidl2016}.

The structure of the paper is as follows.
In the following section we provide an analysis of the discrete total variation of the normal \eqref{eq:tv_of_normal_discrete} and its properties. 
We also compare it to geometric functionals appearing elsewhere in the literature.
In particular, we provide a numerical comparison between \eqref{eq:tv_of_normal_discrete} and surface regularization for a mesh denoising problem.
\Cref{sec:discrete_split_Bregman} is devoted to the formulation of an ADMM method which generalizes the split Bregman algorithm to the manifold-valued problem \eqref{eq:geometric_inverse_problem_with_TV_of_normal}.
In \cref{sec:implementation_details}, we describe an inclusion detection problem of type \eqref{eq:geometric_inverse_problem_with_TV_of_normal}, motivated by geophysical applications.
We also provide implementation details in the finite element framework \fenics.
Corresponding numerical results are presented in \Cref{sec:numerical_results}.

%------------------------------------------------------------------
\section{Discrete Total Variation of the Normal}
\label{sec:discrete_tv_of_normal}
%------------------------------------------------------------------

From this section onwards we assume that $\Gamma_h \subset \R^3$ is a piecewise flat, compact, orientable surface without boundary, which consists of a finite number of flat facets with straight sided edges between facets.
Consequently, $\Gamma_h$ can be thought of as a mesh consisting of polyhedral cells with a consistently oriented outer unit normal.
We also assume this mesh to be geometrically conforming, i.e., there are no hanging nodes.
A frequent situation is that $\Gamma_h$ is the boundary mesh of a geometrically conforming volume mesh with polyhedral cells, representing a volume domain $\Omega_h \subset \R^3$.
In our numerical example in \Cref{sec:numerical_results}, we will utilize a volume mesh consisting of tetrahedra, whose surface mesh consists of triangles; see \Cref{fig:volume_and_boundary_meshes}.

\begin{figure}[htp]
	\centering
	\includegraphics[width=0.40\textwidth]{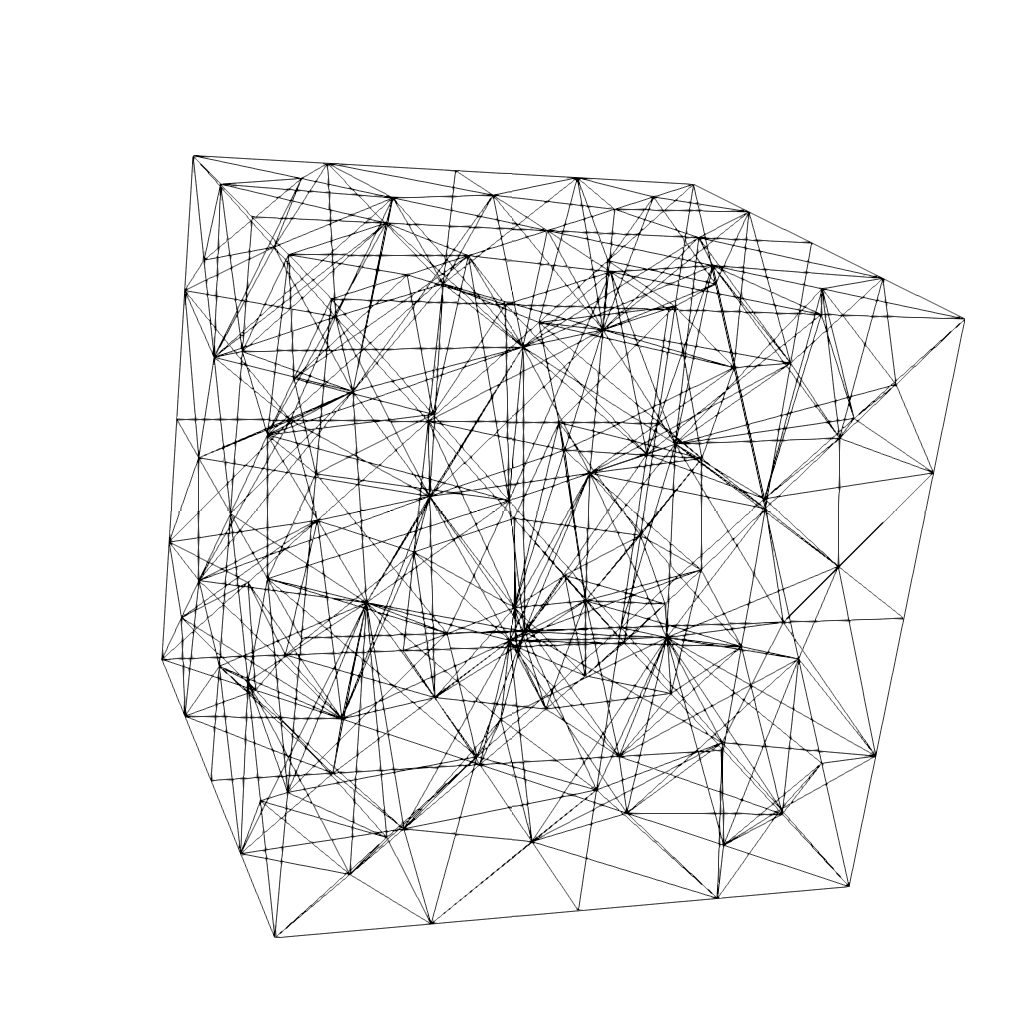}
	\includegraphics[width=0.40\textwidth]{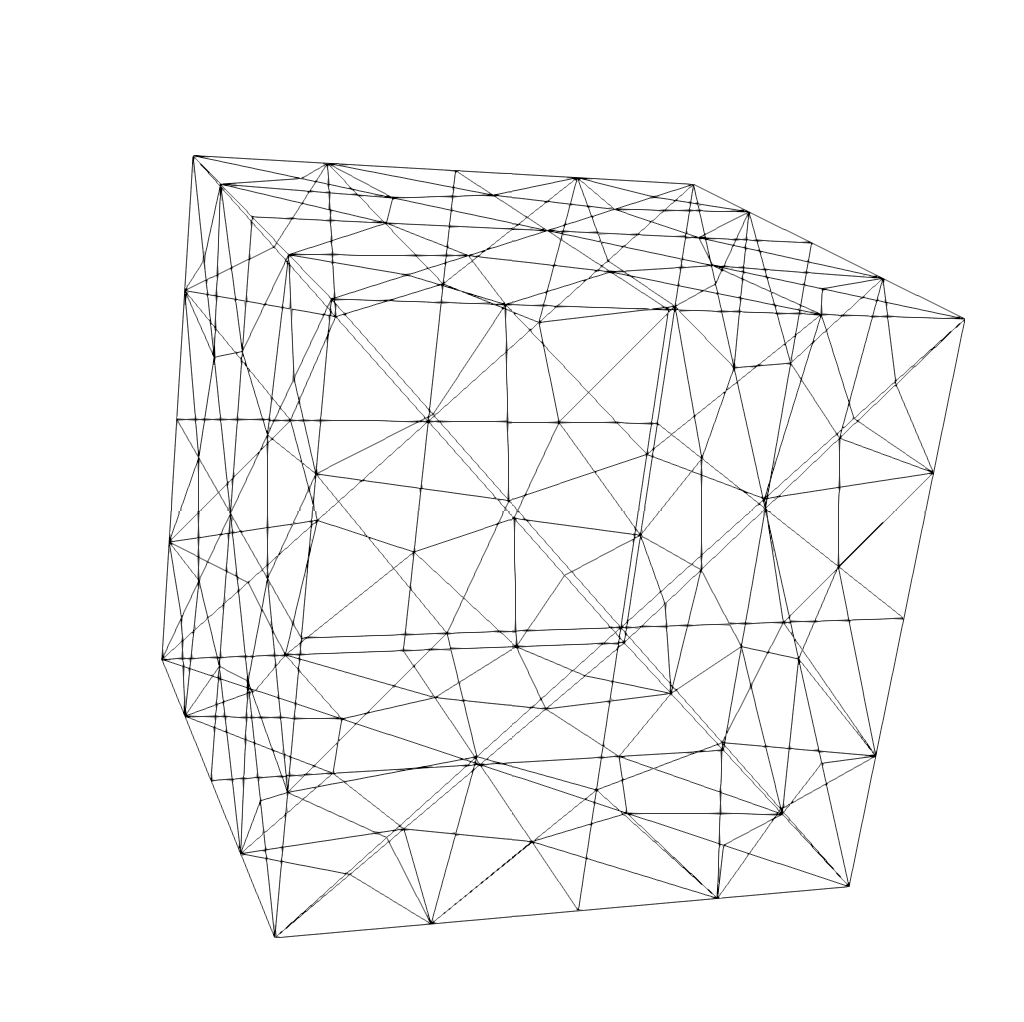}
	\caption{Volume mesh of a cube domain $\Omega_h$ consisting of tetrahedra (left) and corresponding triangular mesh of the boundary $\Gamma_h$ (right).}
	\label{fig:volume_and_boundary_meshes}
\end{figure}

Since the surface $\Gamma_h$ is non-smooth, the definition \eqref{eq:TV_of_normal} of the total variation of the normal proposed in the companion paper \cite{BergmannHerrmannHerzogSchmidtVidalNunez2019:1_preprint} for smooth surfaces does not apply.
Since the normal vector field $\bn$ is piecewise constant here, its variation is concentrated in spontaneous changes across edges between facets, rather than gradual changes expressed by the derivative $D_\Gamma \bn$.
We therefore propose to replace \eqref{eq:TV_of_normal} by 
\begin{equation}
	\label{eq:tv_of_normal_discrete_repeated}
	\abs{\bn}_{DTV(\Gamma_h)}
	\coloneqq
	\sum_E d(\bn_E^+,\bn_E^-) \abs{E}
	,
\end{equation}
where $E$ denotes an edge of Euclidean length $\abs{E}$ between facets.
Each edge has an arbitrary but fixed orientation, so that its two neighboring facets can be addressed as $F_E^+$ and $F_E^-$.
The normal vectors, constant on each facet, are denoted by $\bn_E^+$ and $\bn_E^-$.
Moreover, 
\begin{equation}
	\label{eq:geodesic_distance}
	d(\bn_E^+,\bn_E^-)
	=
	\arccos \bigh(){(\bn_E^+)^\top \bn_E^-}
	=
	\sphericalangle \bigh(){\bn_E^+, \bn_E^-}
\end{equation}
denotes the geodesic distance on $\sphere{2}$, i.e., the angle between the two unit vectors $\bn_E^+$ and $\bn_E^-$; see also \Cref{fig:CoNormal}.

To motivate the definition \eqref{eq:tv_of_normal_discrete_repeated}, consider a family of smooth approximations $\Gamma_\varepsilon$ of the piecewise flat surface $\Gamma_h$.
The approximations are supposed to be of class $C^2$ such that the flat facets are preserved up to a collar of order $\varepsilon$, and smoothing occurs in bands of width $2 \varepsilon$ around the edges.
Such an approximation can be constructed, for instance, by a level-set representation of $\Gamma_h$ by means of a signed distance function~$\Phi$.
Then a family of smooth approximations $\Gamma_\varepsilon$ can be obtained as zero level sets of mollifications $\Phi \circledast \varphi_\varepsilon$ for sufficiently small $\varepsilon$.
Here $\varphi_\varepsilon$ is the standard Friedrichs mollifier in 3D and $\circledast$ denotes convolution.
A construction of this type is used, for instance, in \cite{GomezHernandezLopez2005,BonitoDemlowNochetto2019_preprint}.
An alternative to this procedure is the so-called Steiner smoothing, where $\Gamma_\varepsilon$ is taken to be the boundary of the Minkowski sum of $\Omega_h$ with the ball $B_\varepsilon(0) \subset \R^3$; see for instance \cite[Section~4.4]{Sullivan2008}.

\begin{figure}[htp]
	\centering
	\includegraphics[width=0.45\textwidth]{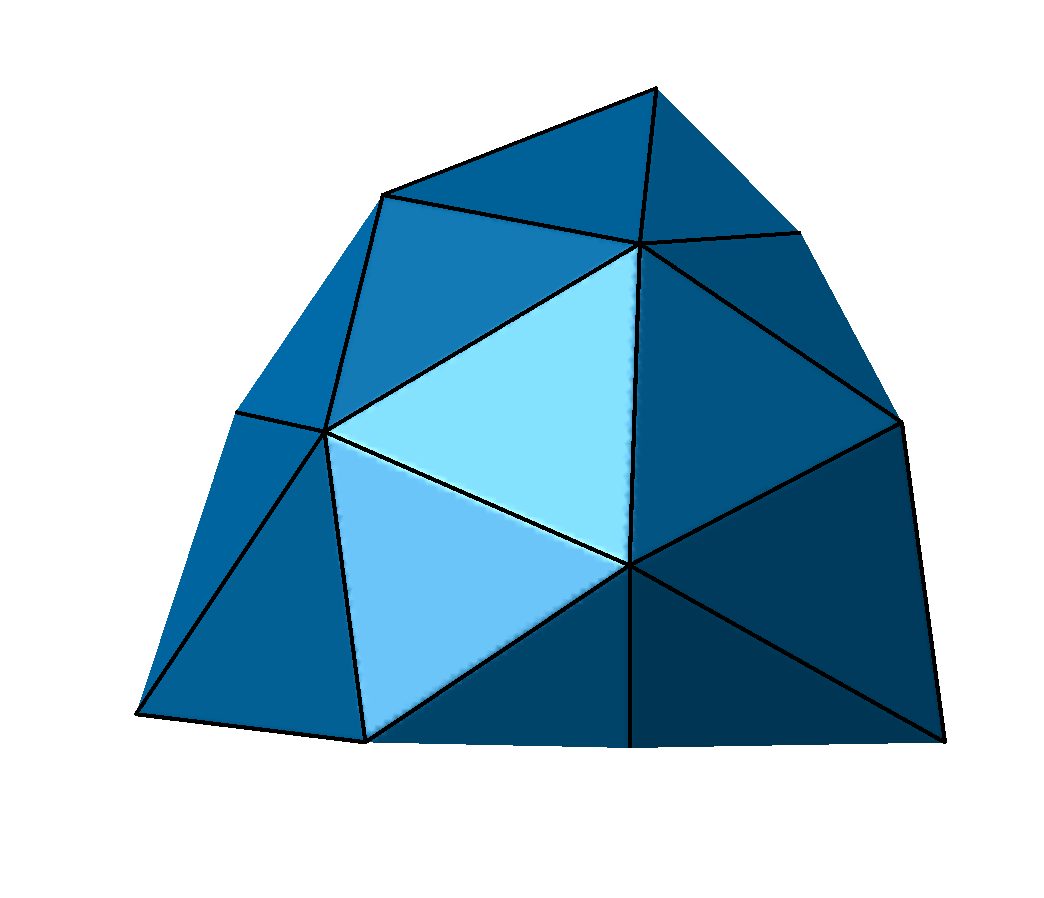}
	\hfill
	\includegraphics[width=0.45\textwidth]{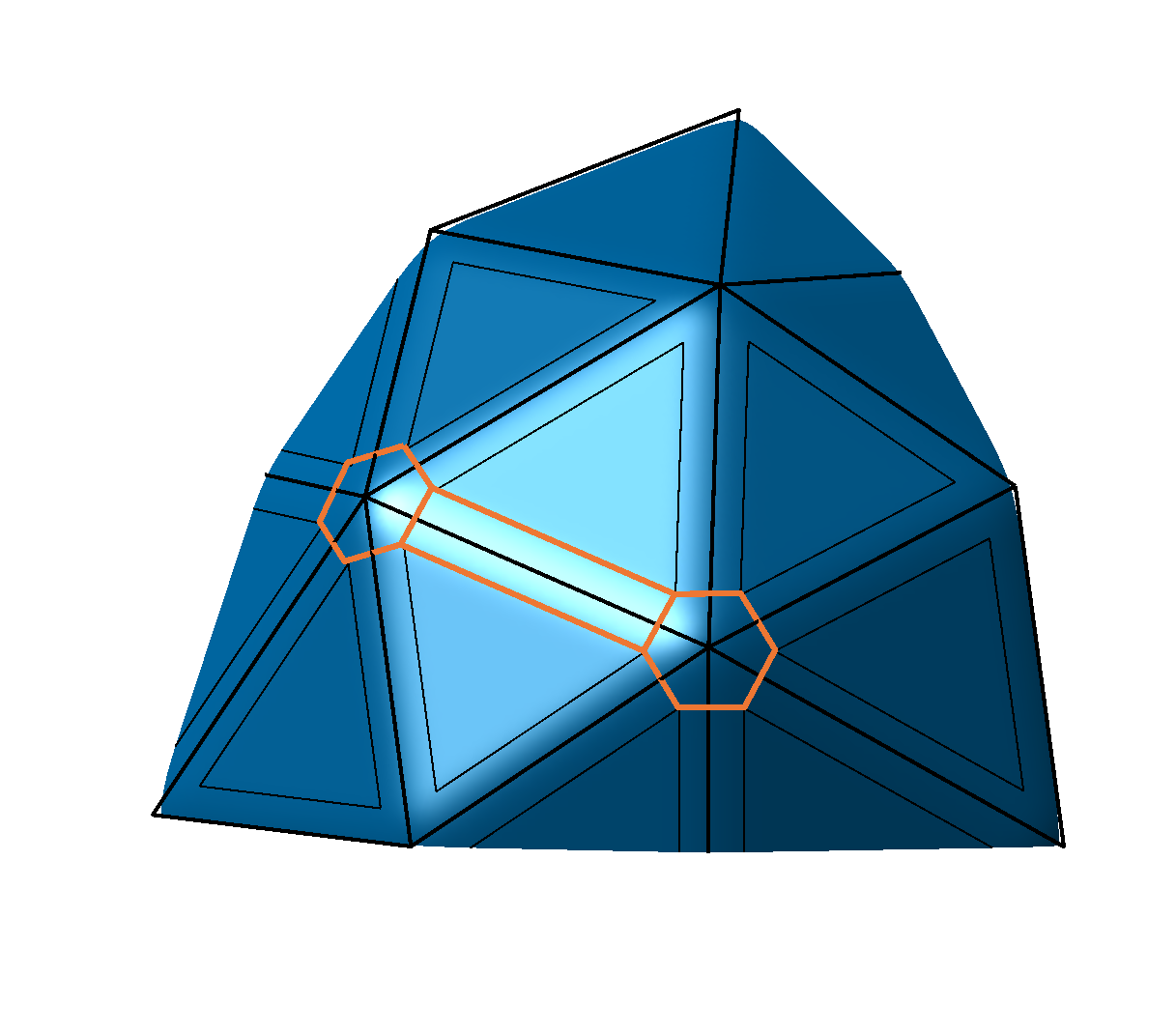}
	\caption{Illustration of the approximation of a portion of a triangulated surface $\Gamma_h$ (left) by a family of smooth surfaces $\Gamma_\varepsilon$ (right). Two vertex caps $B_{V,\varepsilon}$ and one transition region along an edge $I_{E,\varepsilon}$ are highlighted, see the proof of \Cref{theorem:approximation_by_mollification}.}
	\label{fig:approximation_by_smooth_surfaces}
\end{figure}

\begin{theorem}
	\label{theorem:approximation_by_mollification}
	Let $\{\Gamma_\varepsilon\}$ denote a family of smooth approximations of $\Gamma_h$ obtained by mollification, with normal vector fields $\bn_\varepsilon$.
	Then 
	\begin{equation}
		\label{eq:tv_of_normal_in_the_limit}
		\abs{\bn_\varepsilon}_{TV(\Gamma_\varepsilon)}
		\to
		\abs{\bn}_{DTV(\Gamma_h)}
		\quad \text{as } \varepsilon \searrow 0
		.
	\end{equation}
\end{theorem}
\begin{proof}
	Let us denote the vertices in $\Gamma_h$ by $V$ and its edges by $E$.
	Since mollification is local, the normal vector is constant in the interior of each facet minus its collar, which is of order $\varepsilon$. 
	Consequently, changes in the normal vector are confined to a neighborhood of the skeleton.
	We decompose this area into the disjoint union $\dot\bigcup_E I_{E,\varepsilon} \, \dot\cup \, \dot\bigcup_V B_{V,\varepsilon}$.
	Here $I_{E,\varepsilon}$ are the transition regions around edge $E$ where the normal vector is modified due to mollification, and $B_{V,\varepsilon}$ are the regions around vertex $V$.
	On $I_{E,\varepsilon}$, we can arrange the basis $\bxi_{1,2}$ to be aligned and orthogonal to $E$ so that 
	\begin{equation*}
		\int_{I_{E,\varepsilon}} \bigh(){\absRiemannian{(D_{\Gamma_\varepsilon} \bn_\varepsilon) \, \bxi_1}^2 + \absRiemannian{(D_{\Gamma_\varepsilon} \bn_\varepsilon) \, \bxi_2}^2}^{1/2} \, \ds
		=
		\int_{I_{E,\varepsilon}} \absRiemannian{(D_{\Gamma_\varepsilon} \bn_\varepsilon) \, \bxi_1} \, \ds
	\end{equation*}
	holds, which can be easily evaluated as an iterated integral.
	In each stripe in $I_{E,\varepsilon}$ perpendicular to $E$, $\bn_\varepsilon$ changes monotonically along the geodesic path between $\bn_E^+$ and $\bn_E^-$, so that the integral along this stripe yields the constant $d(\bn_E^+,\bn_E^-)$.
	Since the length of $I_{E,\varepsilon}$ parallel to $E$ is $\abs{E}$ up to terms of order $\varepsilon$, we obtain
	\begin{equation*}
		\int_{I_{E,\varepsilon}} \bigh(){\absRiemannian{(D_{\Gamma_\varepsilon} \bn_\varepsilon) \, \bxi_1}^2 + \absRiemannian{(D_{\Gamma_\varepsilon} \bn_\varepsilon) \, \bxi_2}^2}^{1/2} \, \ds
		=
		d(\bn_E^+,\bn_E^-) \bigh[]{ \abs{E} + \OO(\varepsilon) }
		.
	\end{equation*}
	The contributions to $\abs{\bn_\varepsilon}_{TV(\Gamma_\varepsilon)}$ from integration over $B_{V,\varepsilon}$ are of order $\varepsilon$ since $\bigh(){\absRiemannian{(D_{\Gamma_\varepsilon} \bn_\varepsilon) \, \bxi_1}^2 + \absRiemannian{(D_{\Gamma_\varepsilon} \bn_\varepsilon) \, \bxi_2}^2}^{1/2}$ is of order $\varepsilon^{-1}$ and the area of $B_{V,\varepsilon}$ is of order $\varepsilon^2$.
	This yields the claim.
\end{proof}

%------------------------------------------------------------------
\subsection{Comparison with Prior Work for Discrete Surfaces}
\label{subsec:comparison_discrete_surfaces}
%------------------------------------------------------------------

The functional \eqref{eq:tv_of_normal_discrete_repeated} has been used previously in the literature.
We mention that it fits into the framework of total variation of manifold-valued functions defined in \cite{GiaquintaMucci2007,LellmannStrekalovskiyKoetterCremers2013}.
Specifically in the context of discrete surfaces, we mention \cite{Sullivan2006} where the term $H_E \coloneqq \abs{E} \, \Theta_E$ appears as the \emph{total mean curvature} of the edge~$E$.
Here $\Theta_E$ is the exterior dihedral angle, which agrees with $d(\bn_E^+,\bn_E^-)$, see \eqref{eq:geodesic_distance}.
Consequently, \eqref{eq:tv_of_normal_discrete_repeated} can be written as $\sum_E H_E$.
Moreover, \eqref{eq:tv_of_normal_discrete_repeated} appears as a regularizer in \cite{WuZhengCaiFu2015} within a variational model for mesh denoising but the geodesic distances are approximated for the purpose of numerical solution.
We also mention the recent \cite{PellisKilianDellingerWallnerPottmann2019} where \eqref{eq:tv_of_normal_discrete_repeated} appears as a measure of visual smoothness of discrete surfaces.
Particular emphasis is given to the impact of the mesh connectivity.
In our study, the mesh connectivity will remain fixed and only triangular surface meshes are considered in the numerical experiments.

In addition, we are aware of \cite{ZhangWuZhangDeng2015,ZhongXieWangLiuLiu2018}, where
\begin{equation}
	\label{eq:tv_of_normal_discrete_Euclidean}
	\sum_E \abs{\bn_E^+-\bn_E^-}_2 \abs{E},
\end{equation}
was proposed in the context of variational mesh denoising.
Notice that in contrast to \eqref{eq:tv_of_normal_discrete_repeated}, \eqref{eq:tv_of_normal_discrete_Euclidean} utilizes the Euclidean as opposed to the geodesic distance between neighboring normals and is therefore an underestimator for \eqref{eq:tv_of_normal_discrete_repeated}.

Once again, we are not aware of any work in which \eqref{eq:tv_of_normal_discrete_repeated} or its continuous counterpart \eqref{eq:TV_of_normal} were used as a prior in shape optimization or geometric inverse problems involving partial differential equations.

%------------------------------------------------------------------
\subsection{Properties of the Discrete Total Variation of the Normal}
\label{subsec:minimizers_discrete}
%------------------------------------------------------------------

In this section we investigate some properties of the discrete total variation of the normal.
As can be seen directly from \eqref{eq:tv_of_normal_discrete_repeated}, a scaling in which $\Gamma_h$ is replaced by $\scale \Gamma_h$ for some $\scale > 0$ yields 
\begin{equation*}
	\abs{\bn_\scale}_{DTV(\scale \Gamma_h)} 
	= 
	\scale \, \abs{\bn}_{DTV(\Gamma_h)}
	.
\end{equation*}
This is the same behavior observed, e.g., for the total variation of scalar functions defined on two-dimensional domains.
Consequently, when studying optimization problems involving \eqref{eq:tv_of_normal_discrete_repeated}, we need to take precautions to avoid that $\Gamma_h$ degenerates to a point.
This can be achived either by imposing a constraint, e.g., on the surface area, or by considering tracking problems in which an additional loss term appears.

%------------------------------------------------------------------
\subsubsection{Simple Minimizers of the Discrete Total Variation of the Normal}
\label{subsubsec:Minimizers_DTV_exact}
%------------------------------------------------------------------

In this section, we investigate minimizers of $\abs{\bn}_{DTV(\Gamma_h)}$ subject to an area constraint.
More precisely, we consider the following problem.
Given a triangulated surface mesh consisting of vertices $V$, edges $E$ and facets $F$, find the mesh with the same connectivity, which 
\begin{equation}
	\label{eq:minimization_discrete}
	\text{minimizes} 
	\quad 
	\sum_E d(\bn_E^+,\bn_E^-) \abs{E} 
	\quad
	\text{subject to} 
	\quad 
	\sum_F \abs{F} = A_0
	.
\end{equation}
To the best of our knowledge, a precise characterization of the minimizers of~\eqref{eq:minimization_discrete} is an open problem and the solution depends on the connectivity; compare the observations in \cite[Section~4]{PellisKilianDellingerWallnerPottmann2019}.
That is, different triangulations of the same (initial) mesh, e.g., a cube, may yield different minimizers.
We also refer the reader to \cite{AlexaWardetzky2011} for a related observation in discrete mean curvature flow.

We do have, however, the following partial result.
For the proof, we exploit that \eqref{eq:tv_of_normal_discrete_repeated} coincides with the discrete total mean curvature and utilize results from discrete differential geometry.
The reader may wish to consult \cite{MeyerDesbrunSchroederBarr2003,Polthier2005,Wardetzky2006,BobenkoSpringborn2007,CraneDeGoesDesbrungSchroeder2013}.
\begin{theorem}
	\label{theorem:who_is_stationary_discrete}
	The icosahedron and the cube with crossed diagonals are stationary for \eqref{eq:minimization_discrete} within the class of triangulated surfaces $\Gamma_h$ of constant area and identical connectivity.
\end{theorem}
\begin{proof}
	Let us consider the Lagrangian associated with \eqref{eq:minimization_discrete}, 
	\begin{equation}
		\label{eq:minimization_discrete_Lagrangian}
		\LL(\bx_1,\ldots,\bx_{\nvertices},\mu) 
		\coloneqq
		\sum_E d(\bn_E^+,\bn_E^-) \abs{E} 
		+ 
		\mu \, \Big( \sum_F \abs{F} - A_0 \Big)
		.
	\end{equation}
	Here $\bx_i \in \R^3$ denote the coordinates of vertex~$\#i$ and $\nvertices$ is the total number of vertices of the triangular surface mesh.
	Notice that the normal vectors $\bn_E^\pm$, edge lengths $\abs{E}$ and facet areas $\abs{F}$ depend on these coordinates.
	The gradient of \eqref{eq:minimization_discrete_Lagrangian} w.r.t.\ $\bx_i$ can be represented as
	\begin{multline}
		\label{eq:minimization_discrete_Lagrangian_partial_gradient}
		\nabla_{\bx_i} \LL(\bx_1,\ldots,\bx_{\nvertices},\mu) 
		=
		\sum_{j \in \NN(i)} \Big[ \frac{d(\bn_{E_{ij}}^+,\bn_{E_{ij}}^-)}{\abs{E_{ij}}} + \frac{\mu}{2} \big( \cot \alpha_{ij} + \cot \beta_{ij} \big) \Big] (\bx_i - \bx_j) 
		,
	\end{multline}
	see for instance \cite{CraneDeGoesDesbrungSchroeder2013}.
	Here $\NN(i)$ denotes the index set of vertices adjacent to vertex~$\#i$.
	For any $j \in \NN(i)$, $E_{ij}$ denotes the edge between vertices~$\#i$ and $\#j$.
	Moreover, $\alpha_{ij}$ and $\beta_{ij}$ are the angles as illustrated in \Cref{fig:angles}.

	For the icosahedron with surface area $A_0$, all edges have length $\abs{E_{ij}} = \big( \frac{A_0}{5 \, \sqrt{3}} \big)^{1/2}$.
	Moreover, since all facets are unilaterial triangles, $\alpha_{ij} = \beta_{ij} = \pi/3$ holds.
	Finally, the exterior dihedral angles $d(\bn_{E_{ij}}^+,\bn_{E_{ij}}^-)$ are all equal to $\arccos(\sqrt{5}/3) \approx 41.81^\circ$.
	Consequently, the Lagrangian is stationary for the Lagrange multiplier $\mu = - \sqrt{3} \arccos(\sqrt{5}/3) \big( \frac{5 \, \sqrt{3}}{A_0} \big)^{1/2}$.

	We remark that \eqref{eq:tv_of_normal_discrete_repeated} and thus \eqref{eq:minimization_discrete_Lagrangian_partial_gradient} is not differentiable when one or more of the angles $d(\bn_{E_{ij}}^+,\bn_{E_{ij}}^-)$ are zero.
	This is the case for the cube with crossed diagonals, see \Cref{fig:angles}.
	However, the right hand side in \eqref{eq:minimization_discrete_Lagrangian_partial_gradient} still provides a generalized derivative of $\LL$ in the sense of Clarke.
	\\
	In contrast to the icosahedron, the cube has two types of vertices.
	When $\bx_i$ is the center vertex of one of the lateral surfaces, then $d(\bn_{E_{ij}}^+,\bn_{E_{ij}}^-) = 0$ and $\alpha_{ij} = \beta_{ij} = \pi/4$ for all $j \in \NN(i)$.
	Moreover, since $\sum_{j \in \NN(i)} (\bx_i - \bx_j) = \bnull$ holds, $\bnull$ is an element of the generalized (partial) differential of $\LL$ at $(\bx_1,\ldots,\bx_{\nvertices},\mu)$ w.r.t.\ $\bx_i$, independently of the value of the Lagrange multiplier $\mu$.
	Now when $\bx_i$ is a vertex of \eqq{corner type}, we need to distinguish two types of edges.
	Along the three edges leading to neighbors of the same type, we have a exterior dihedral angle of $d(\bn_{E_{ij}}^+,\bn_{E_{ij}}^-) = \pi/2$, length $\abs{E_{ij}} = (A_0/6)^{1/2}$ and $\alpha_{ij} = \beta_{ij} = \pi/2$.
	Along the three remaining edges leading to surface centers, we have $d(\bn_{E_{ij}}^+,\bn_{E_{ij}}^-) = 0$ and $\alpha_{ij} = \beta_{ij} = \pi/4$.
	Thus for vertices of \eqq{corner type}, it is straightforward to verify that $\bnull$ belongs to the generalized (partial) differential of $\LL$ at $(\bx_1,\ldots,\bx_{\nvertices},\mu)$ w.r.t.\ $\bx_i$ if 
	\begin{equation*}
		\left( \frac{\pi \, \sqrt{2}/2}{(A_0/6)^{1/2}} + 2 \, \mu \right) 
		\begin{pmatrix}
			1 \\ 1 \\ 1
		\end{pmatrix}
		= 
		\bnull
	\end{equation*}
	holds, which is true for the obvious choice of $\mu$.
\end{proof}
Numerical experiments indicate that the icosahedron as well as the cube are not only stationary points, but also local minimizers of \eqref{eq:minimization_discrete}.
We can thus conclude that the discrete objective \eqref{eq:tv_of_normal_discrete_repeated} exhibits different minimizers than its continuous counterpart \eqref{eq:TV_of_normal} for smooth surfaces.
In particular, \eqref{eq:tv_of_normal_discrete_repeated} admits and promotes piecewise flat minimizers such as the cube.
This is in accordance with observations made in \cite[Section~3.2]{PellisKilianDellingerWallnerPottmann2019} that optimal meshes typically exhibit a number of zero dihedral angles.
This property sets our functional apart from other functionals previously used as priors in shape optimization and geometric inverse problems.
For instance, the popular surface area prior is well known to produce smooth shapes; see the numerical experiments in \cref{subsubsec:DTV_vs_area} below.
\begin{figure}[h]
	\centering
	\includegraphics[width=0.3\textwidth]{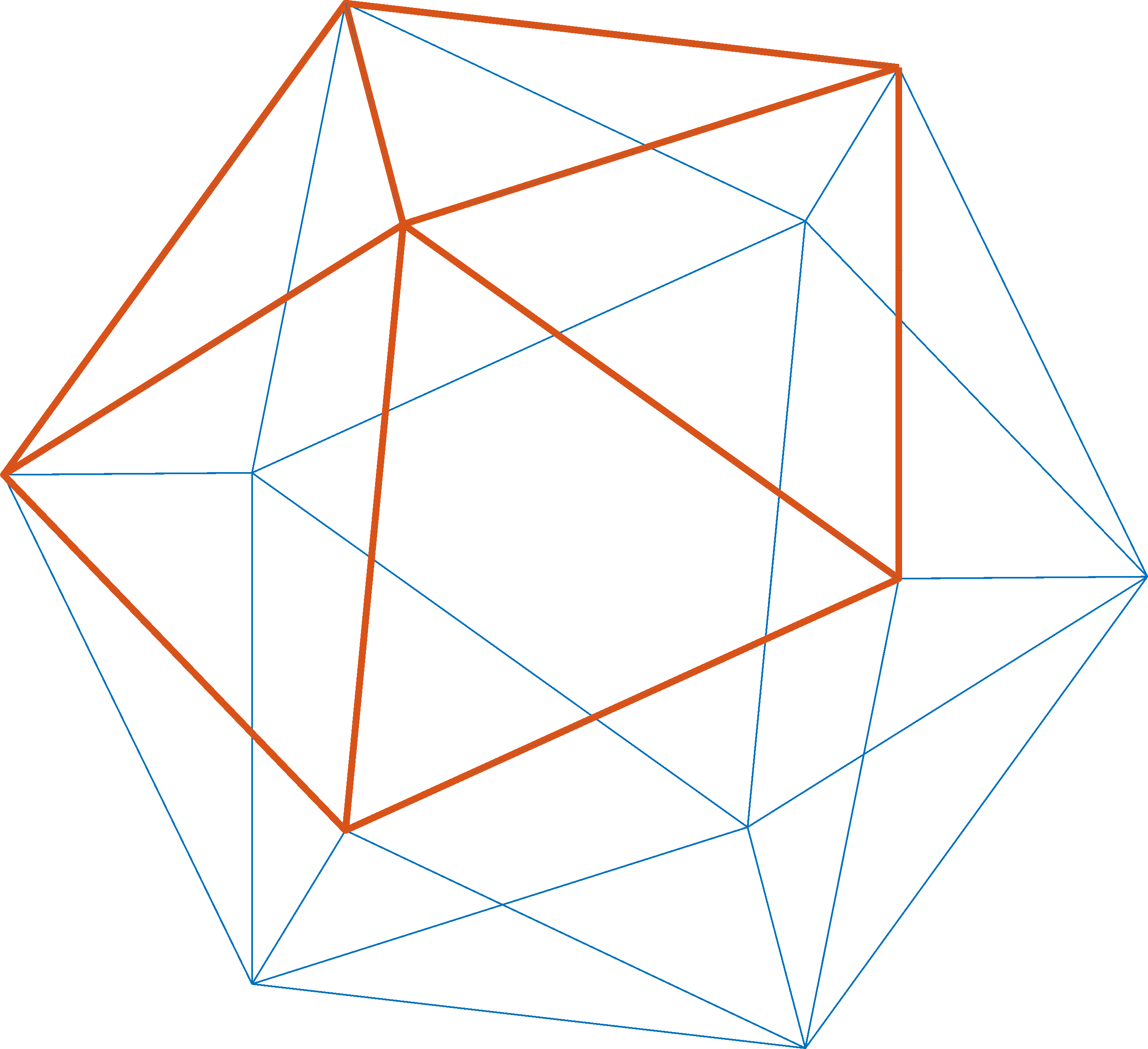}
	\hfill
	\includegraphics[width=0.3\textwidth]{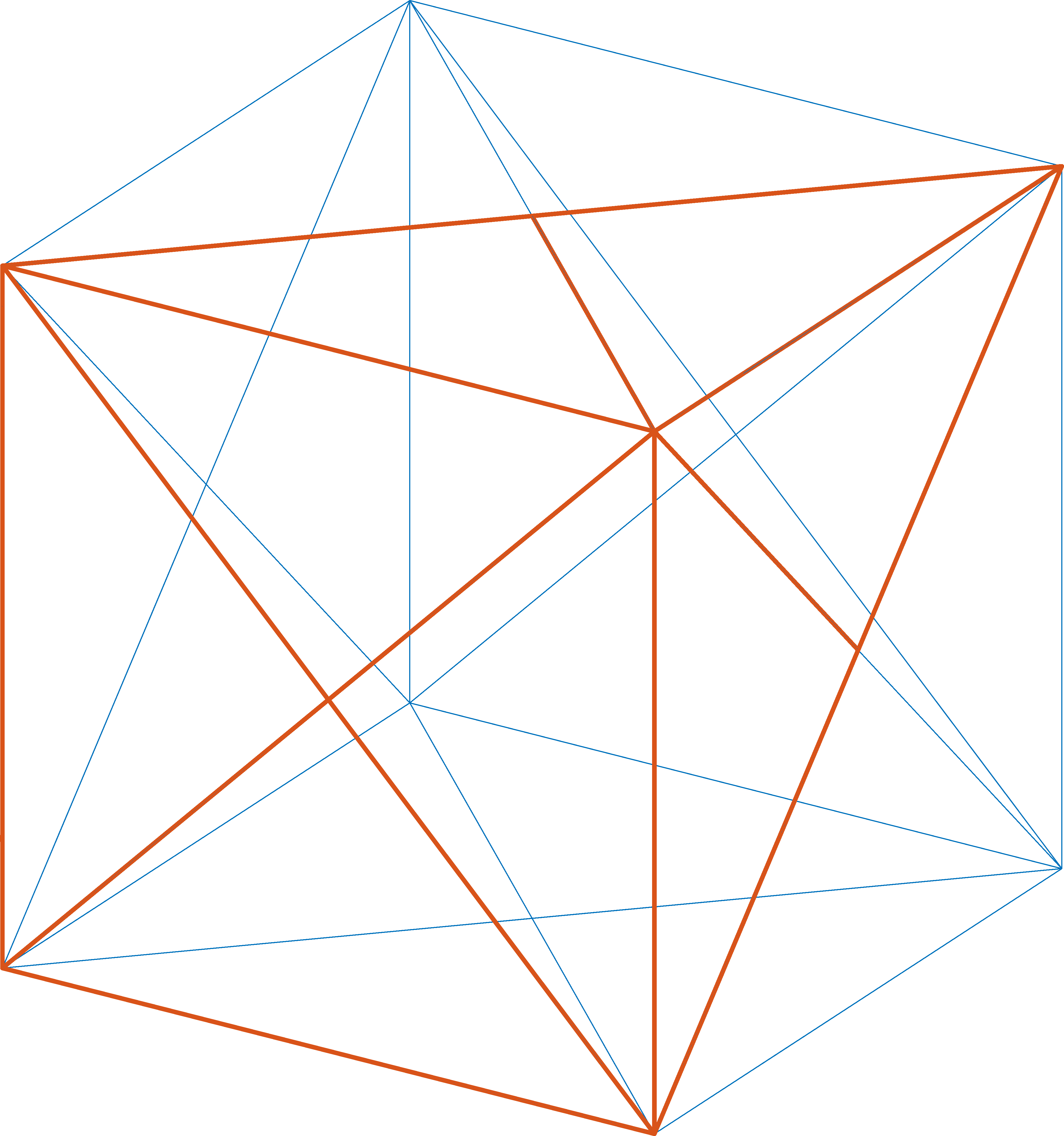}
	\hfill
	\begin{overpic}[width=0.3\textwidth]{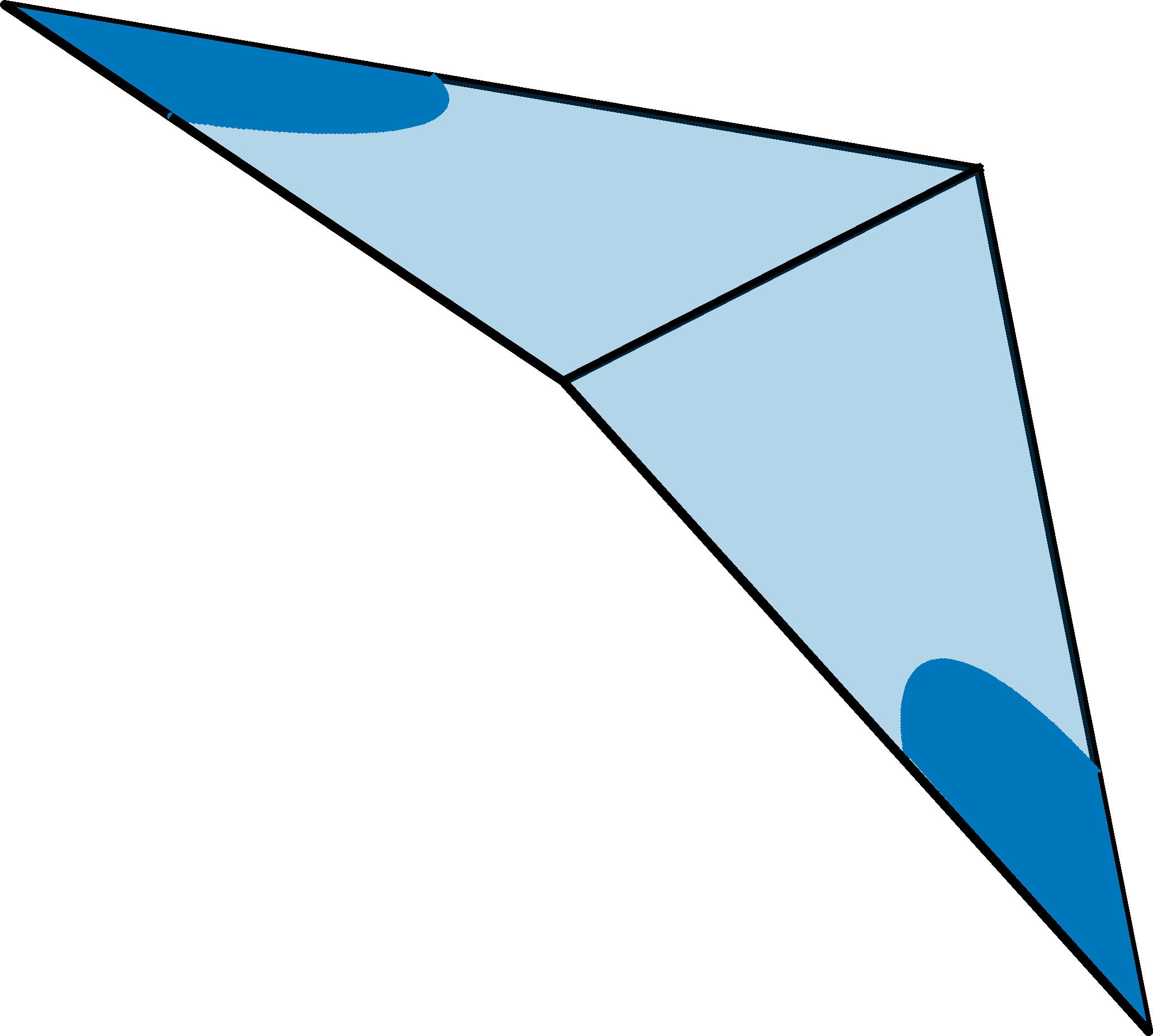}
		\put(14,82){$\alpha_{ij}$}
		\put(81,23){$\beta_{ij}$}
		\put(60,52){\rotatebox{27}{$E_{ij}$}}
	\end{overpic}
	\caption{The icosahedron and the cube with crossed diagonals, two stationary surfaces for \eqref{eq:minimization_discrete}. The highlighted regions as well as the figure on the right illustrate the proof of \Cref{theorem:who_is_stationary_discrete}.}
	\label{fig:angles}
\end{figure}

\begin{table}
	\centering
	\begin{tabular}{llll}
		\toprule
		& unit cube               & tetrahedron       & icosahedron       \\
		\midrule
		edge length $\abs{E}$     & 1                       & $\approx 1.8612$  & $\approx 0.8324$  \\
		number of edges           & 12                      & 6                 & 30                \\
		exterior dihedral angle   & $\pi/2$                 & $\approx 1.9106$  & $\approx 0.7297$  \\
		\midrule
		$\abs{\bn}_{DTV(\Gamma_h)}$ & $6 \pi \approx 18.8496$ & $\approx 21.3365$ & $\approx 18.2218$ \\
		\bottomrule \\
	\end{tabular}
	\caption{Values of the discrete total variation of the normal functional \eqref{eq:tv_of_normal_discrete_repeated} for the cube with edge length~1, as well as the regular tetrahedron and the icosahedron with the same surface area as the cube.}
	\label{tab:selected_DTV_values}
\end{table}

%------------------------------------------------------------------
\subsubsection{Comparison of Discrete and Continuous Total Variation of the Normal}
\label{subsubsec:}
%------------------------------------------------------------------

In this section we compare the values of \eqref{eq:TV_of_normal} and \eqref{eq:tv_of_normal_discrete_repeated} for a sphere~$\Gamma$, and a sequence of discretized spheres~$\Gamma_h$.
For comparison, we choose $\Gamma$ to have the same surface area as the cube in the previous section, i.e., we use $r = \sqrt{3/(2\pi)}$ as the radius.
It is easy to see that since the principal curvatures of a sphere~$\Gamma$ of radius~$r$ are $k_1 = k_2 = 1/r$, \eqref{eq:TV_of_normal} becomes
\begin{equation*}
	\abs{\bn}_{TV(\Gamma)}
	=
	\int_\Gamma \bigh(){k_1^2 + k_2^2}^{1/2} \, \ds
	=
	4 \, \pi \, r^2 \frac{\sqrt{2}}{r}
	=
	4 \, \sqrt{2} \, \pi \, r
	,
\end{equation*}
which amounts to $4 \sqrt{3 \pi} \approx 12.2799$ for the sphere under consideration.

To compare this to the discrete total variation of the normal, we created a sequence of triangular meshes~$\Gamma_h$ of this sphere with various resolutions using \gmsh\ and evaluated \eqref{eq:tv_of_normal_discrete_repeated} numerically.
The results are shown in \Cref{tab:discrete_objective_sphere}.
They reveal a factor of approximately $\sqrt{2}$ between the discrete and continuous functionals for the sphere.
To explain this discrepancy, recall that the principal curvatures of the sphere are $k_1 = k_2 = 1/r$.
This implies that the derivative map $D_\Gamma \bn$ has rank two everywhere.
Discretized surfaces behave fundamentally different in the following respect.
Their curvature is concentrated on the edges, and one of the principal curvatures (the one in the direction along the edge) is always zero.
So even for successively refined meshes, e.g., of the sphere, one is still measuring only one principal curvature at a time.
We are thus led to the conjecture that the limit of \eqref{eq:tv_of_normal_discrete_repeated} for sucessively refined meshes is the \eqq{anisotropic}, yet still intrinsic measure $\int_\Gamma \abs{k_1} + \abs{k_2} \, \ds$, whose value for the sphere in \Cref{tab:selected_DTV_values} is $4 \sqrt{6 \pi} \approx 17.3664$.
The factor $\sqrt{2}$ can thus be attributed to the ratio between the $\ell_1$- and $\ell_2$-norms of the vector $(1,1)^\top$.
This observation is in accordance with the findings in \cite[Section~1.2]{PellisKilianDellingerWallnerPottmann2019}.

One could consider an \eqq{isotropic} version of \eqref{eq:tv_of_normal_discrete_repeated} in which the dihedral angles across all edges meeting at any given vertex are measured jointly.
These alternatives will be considered elsewhere.

\begin{table}[tbp]
	\centering
	\begin{tabular}{@{}rrrrr@{}}
		\toprule
		$\nvertices$ & $\nedges$ & $\ntriangles$ & $\abs{\bn}_{DTV(\Gamma_h)}$ & $\abs{\bn}_{DTV(\Gamma_h)}/\abs{\bn}_{TV(\Gamma_h)}$ \\
		\midrule
		54& 156& 104& 17.01045& 1.38522\\
		270& 804& 536& 17.47614& 1.42315\\
		871& 2,607& 1,738& 17.34861& 1.41276\\
		1,812& 5,430& 3,620& 17.35852& 1.41357\\
		3,314& 9,936& 6,624& 17.36350& 1.41398\\\midrule
		9,530& 28,584& 19,056& 17.36855& 1.41439\\
		82,665& 247,989& 165,326& 17.37524& 1.41493\\
		101,935& 305,799& 203,866& 17.37341& 1.41478\\
		335,216& 1,005,642& 670,428& 17.37389& 1.41482\\
		958,022& 2,874,060& 1,916,040& 17.37410& 1.41484\\
		\bottomrule \\
	\end{tabular}
	\caption{Various triangulations $\Gamma_h$ of a sphere~$\Gamma$ with radius $r = \sqrt{3/(2\pi)}$, their values of \eqref{eq:tv_of_normal_discrete_repeated} and the ratio between \eqref{eq:tv_of_normal_discrete_repeated} and \eqref{eq:TV_of_normal}. The value of the latter is $\abs{\bn}_{TV(\Gamma)} = 4 \sqrt{3 \pi} \approx 12.2799$. $\nvertices$, $\nedges$ and $\ntriangles$ denote the number of vertices, edges, and triangles of the respective mesh.}
	\label{tab:discrete_objective_sphere}
\end{table}

%------------------------------------------------------------------
\subsubsection{Discrete Total Variation Compared to Surface Area Regularization}
\label{subsubsec:DTV_vs_area}
%------------------------------------------------------------------

In this section we consider a specific instance of the general problem \eqref{eq:geometric_inverse_problem_with_TV_of_normal} and compare our discrete TV functional with the surface area regularizer.
We begin with a triangular surface mesh $\Gamma_h$ of a box $\Omega = (-1,1) \times (-1.5,1.5) \times (-2, 2)$ and add normally distributed noise to the coordinate vector of each vertex in average normal direction of the adjacent triangles with zero mean and standard deviation $\sigma = 0.2$ times the average edge length. 
We denote the noisy vertex positions as $\widetilde \bx_V$ and utilize a simple least-squares functional as our loss function and consider the following mesh denoising problem,
\begin{equation}
	\label{eq:surface_recovery_DTV}
	\begin{aligned}
		& \text{Minimize} \quad \frac{1}{2} \sum_V \abs{\bx_V - \widetilde \bx_V}_2^2 + \beta \, \abs{\bn}_{DTV(\Gamma_h)} \\
		& \text{w.r.t.\ the vertex positions $\bx_V$ of the discrete surface $\Gamma_h$}.
	\end{aligned}
\end{equation}
Here the sum runs over the vertices of~$\Gamma_h$.
For comparison, we also consider a variant
\begin{equation}
	\label{eq:surface_recovery_area}
	\begin{aligned}
		& \text{Minimize} \quad \frac{1}{2} \sum_V \abs{\bx_V - \widetilde \bx_V}_2^2 + \gamma \, \sum_F \abs{F} \\
		& \text{w.r.t.\ the vertex positions $\bx_V$ of the discrete surface $\Gamma_h$},
	\end{aligned}
\end{equation}
where we utilize the total surface area as prior.

A numerical approach to solve the non-smooth problem \eqref{eq:surface_recovery_DTV} will be discussed in \cref{sec:discrete_split_Bregman}.
By contrast, problem \eqref{eq:surface_recovery_area} is a fairly standard smooth discrete shape optimization problem and we solve it using a simple shape gradient descent scheme. 
The details how to obtain the shape derivative and shape gradient are the same as described in \cref{subsec:DiscreteShapeDerivative} for problem \eqref{eq:surface_recovery_DTV}.

\Cref{fig:comparison_TV_surface_area} shows the numerical solutions of \eqref{eq:surface_recovery_DTV} and \eqref{eq:surface_recovery_area} for various choices of the regularization parameters $\beta$ and $\gamma$, respectively. 
The initial guess for both problems is a sphere with the same connectivity as $\Gamma_h$.
We can clearly see that our functional \eqref{eq:tv_of_normal_discrete_repeated} achieves a very good reconstruction of the original shape for a proper choice of $\beta$.
By contrast, the surface area regularization requires a relatively large choice of $\gamma$ in order to reasonably reduce the noise, which in turn leads to a significant shrinkage of the surface and a rounding of the sharp features.
\begin{figure}[h!]
	\centering
	\includegraphics[width=0.3\textwidth]{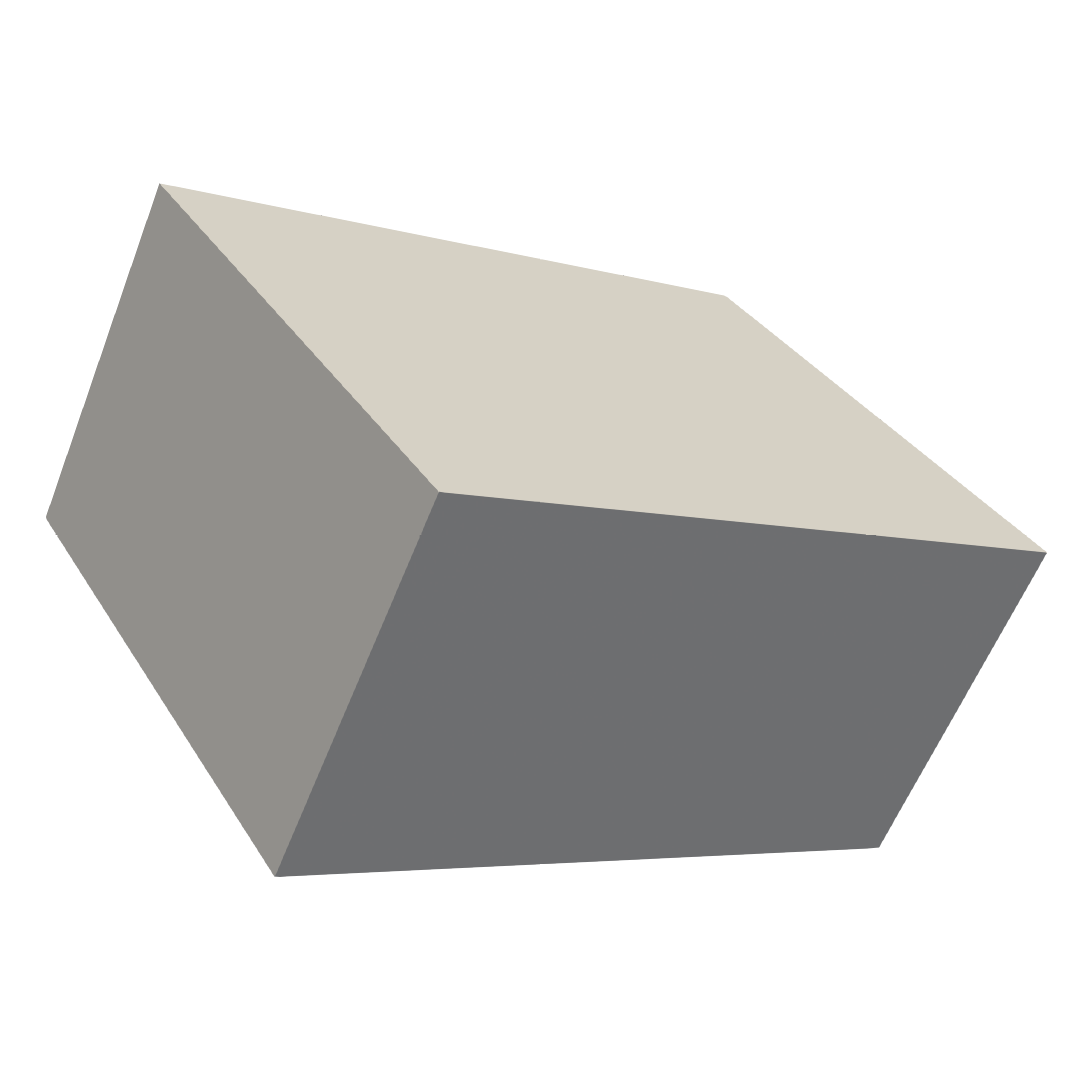}
	\includegraphics[width=0.3\textwidth]{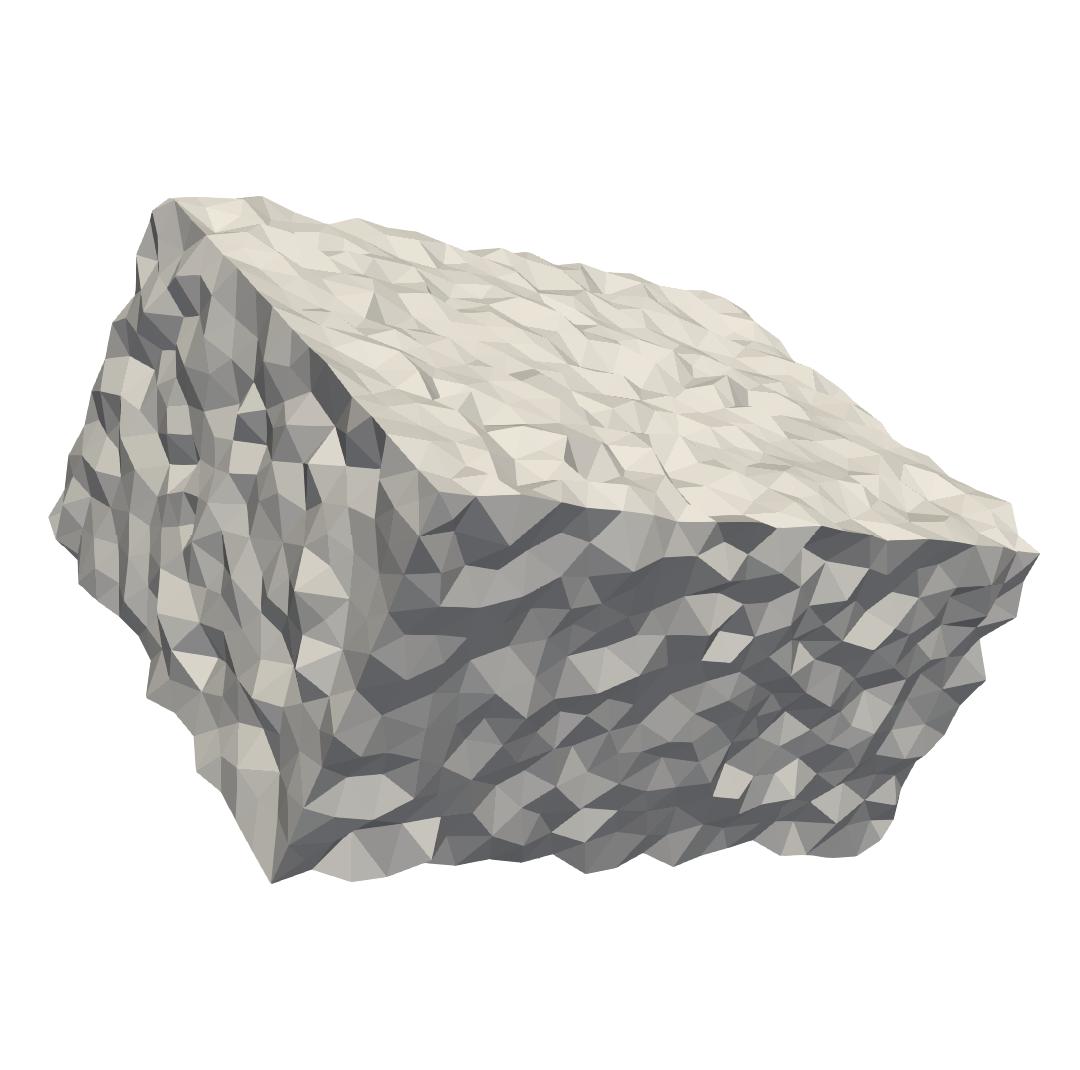}
	\includegraphics[width=0.3\textwidth]{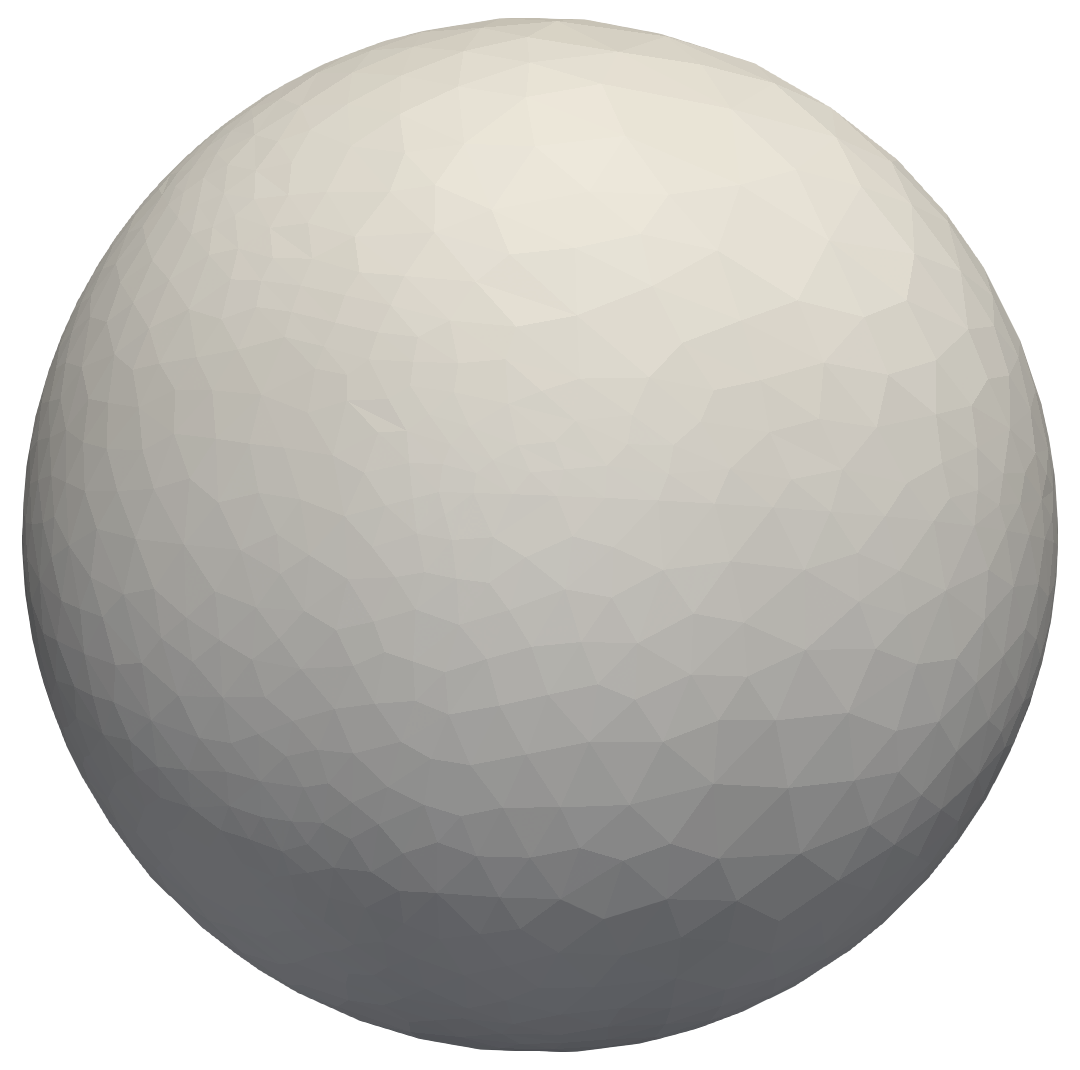}\\
	\includegraphics[width=0.3\textwidth]{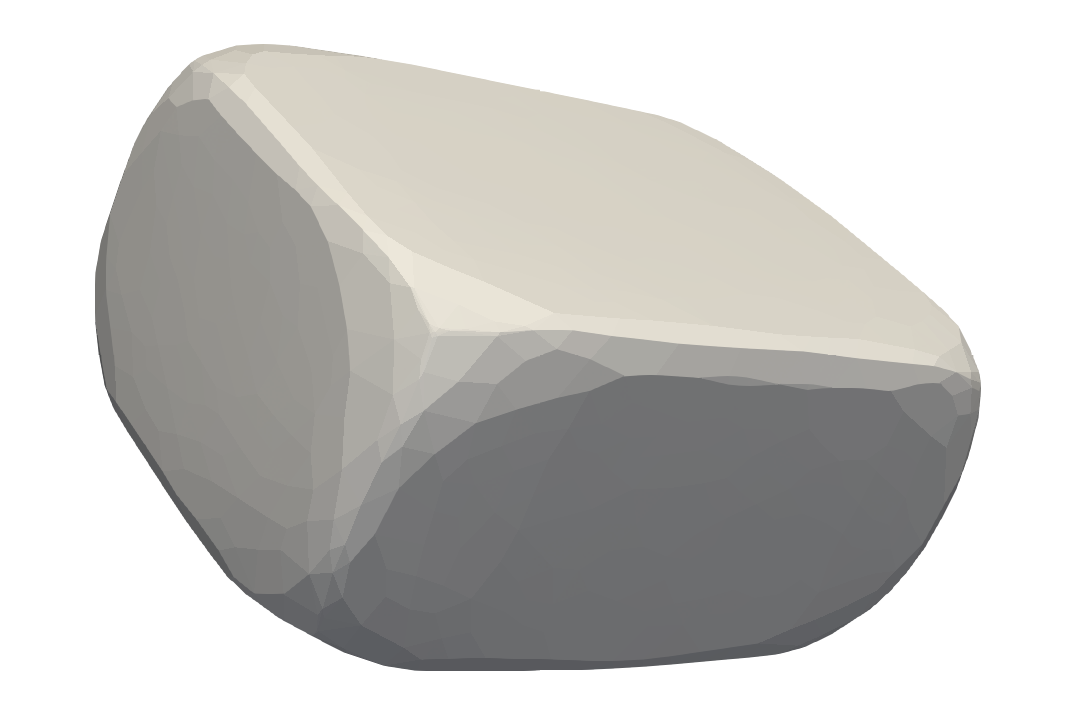}
	\includegraphics[width=0.3\textwidth]{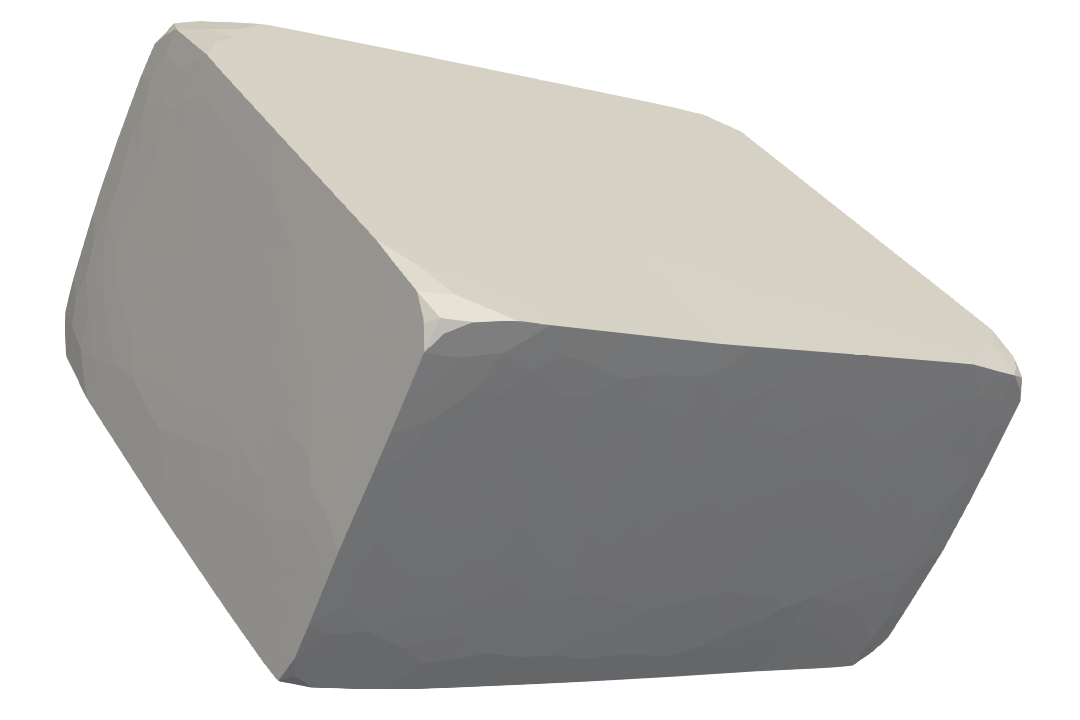}
	\includegraphics[width=0.3\textwidth]{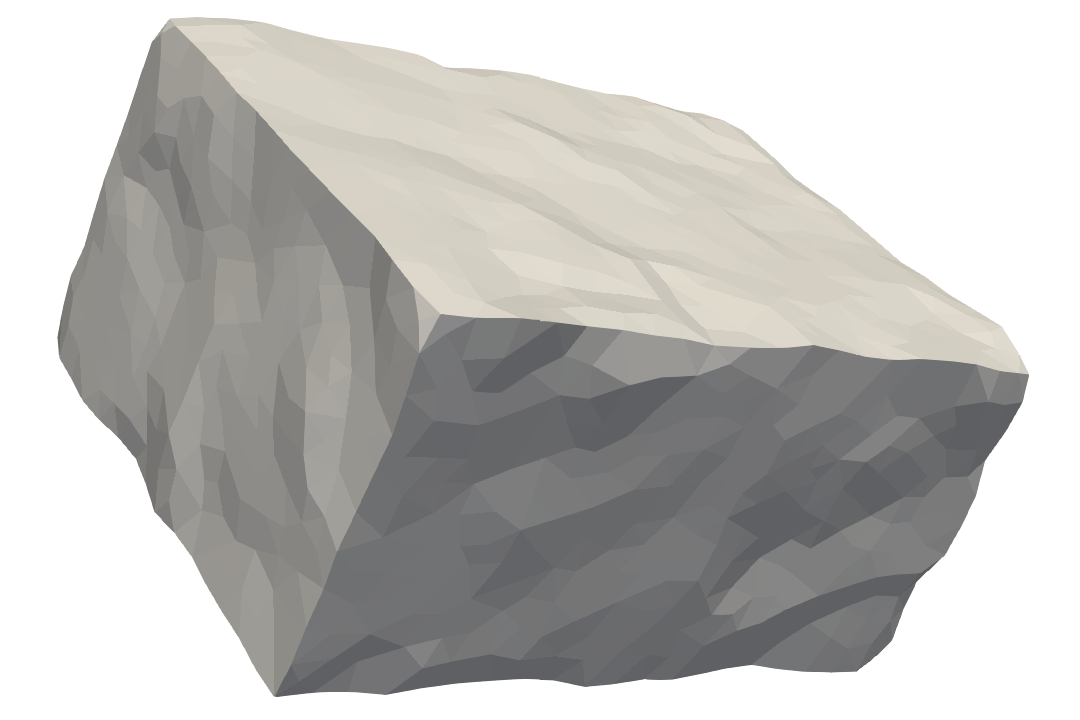}\\
	\includegraphics[width=0.3\textwidth]{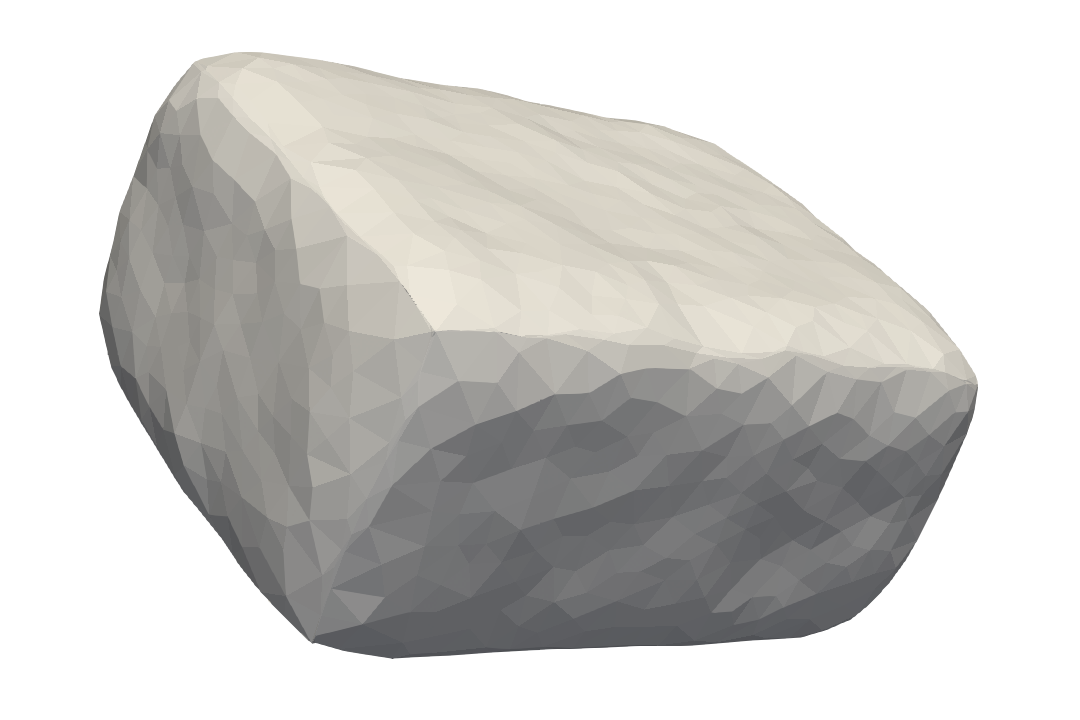}
	\includegraphics[width=0.3\textwidth]{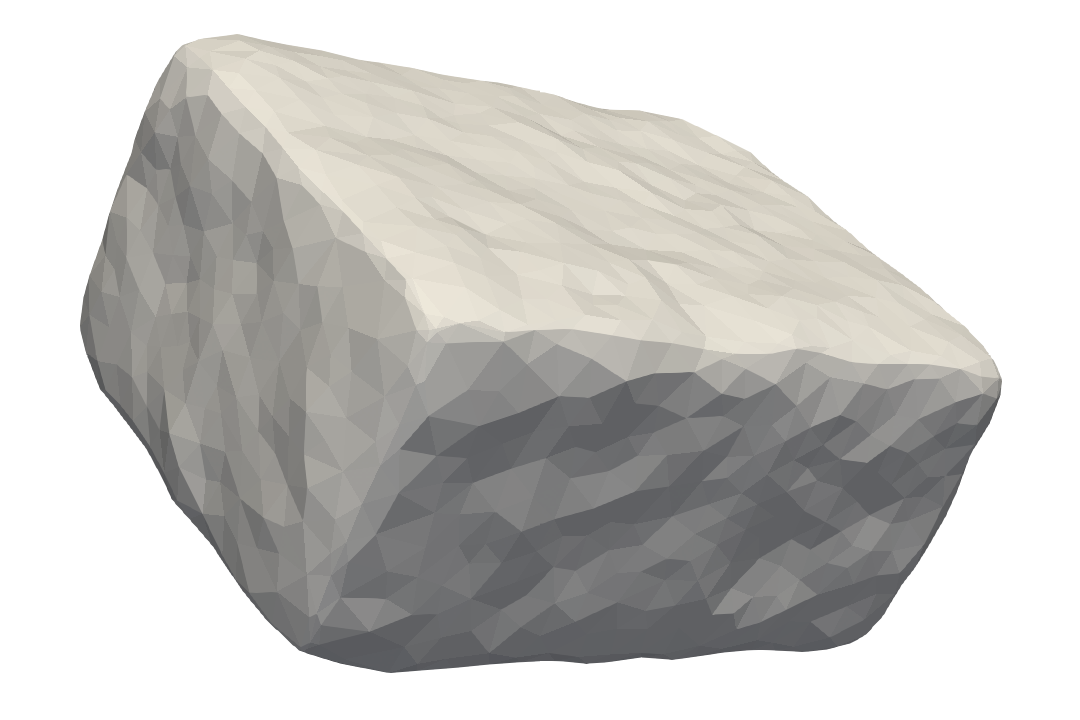}
	\includegraphics[width=0.3\textwidth]{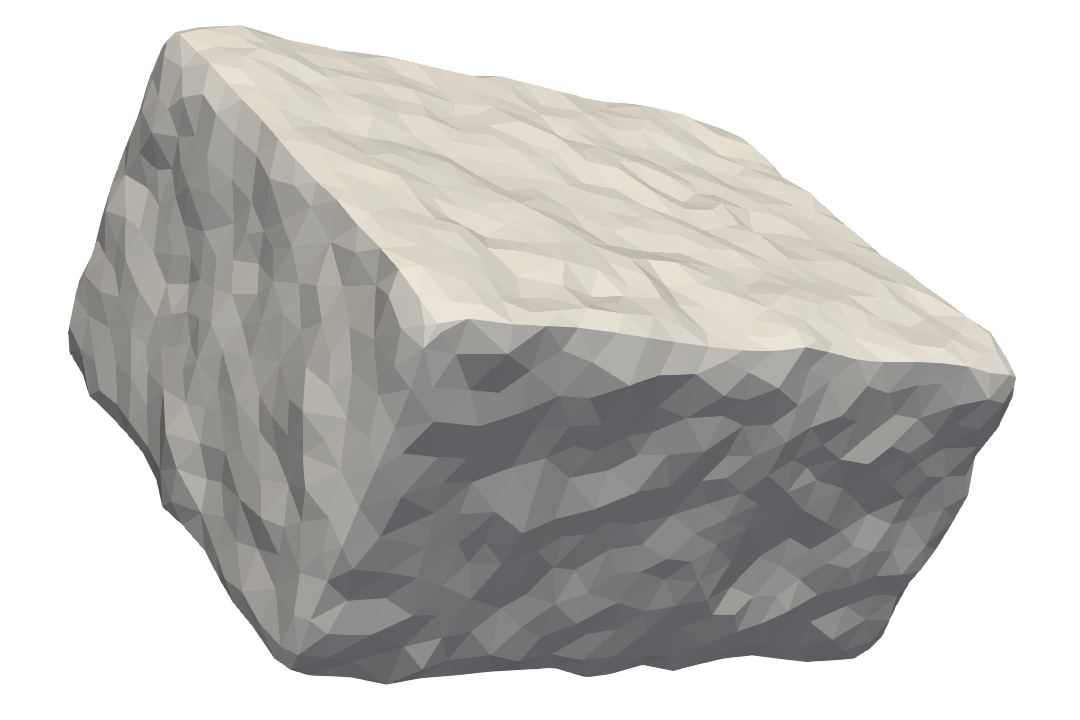}
	\caption{Top row: original box and desired outcome of the noise reduction, noisy box with vertex coordinates~$\widetilde \bx_V$ used for the data fidelity term, and sphere with same connectivity used as the initial guess; middle row: results for total variation of the normal \eqref{eq:surface_recovery_DTV} with $\beta = 10^{-2},\ 10^{-3}, 10^{-4}$; bottom row: results for surface area regularization \eqref{eq:surface_recovery_area} with $\gamma = 0.02,\ 0.01,\ 0.005$.}
	\label{fig:comparison_TV_surface_area}
\end{figure}

%------------------------------------------------------------------
\section{Discrete Split Bregman Iteration}
\label{sec:discrete_split_Bregman}
%------------------------------------------------------------------

In this section, we develop an optimization scheme to solve the non-smooth problem~\eqref{eq:tv_of_normal_discrete}. To this end, we adapt the well-known split Bregman method to our setting. 
This leads to a discrete realization of the approach presented in~\cite[section 4]{BergmannHerrmannHerzogSchmidtVidalNunez2019:1_preprint}. Recall that combining~\eqref{eq:tv_of_normal_discrete} with~\eqref{eq:geometric_inverse_problem_with_TV_of_normal} results in the problem
\begin{equation}
	\label{eq:geometric_inverse_problem_with_TV_of_normal_discrete}
	\begin{aligned}
		& \text{Minimize} \quad \ell(u(\Omega_h),\Omega_h) + \beta \, \sum_E d(\bn_E^+,\bn_E^-) \abs{E} \\
		& \text{w.r.t.\ the vertex positions of $\Omega_h$},
	\end{aligned}
\end{equation}
where $E$ are the edges of the unknown part~$\Gamma_h$ of the boundary $\partial\Omega_h$. We will consider a concrete example in \Cref{subsec:EIT}.

Notice that the second term in the objective in \eqref{eq:geometric_inverse_problem_with_TV_of_normal_discrete} is non-differentiable whenever $\bn_E^+ = \bn_E^-$ occurs on at least one edge.
Following the classical split Bregman approach, we introduce a splitting in which the variation of the normal vector becomes an independent variable.
Since this variation is confined to edges, where the normal vector jumps (without loss of generality) from $\bn_E^+$ to $\bn_E^-$, this new variable becomes
\begin{equation}
	\label{eq:geometric_inverse_problem_split_constraint}
	\bd_E = \mylog{\bn_E^+}{\bn_E^-} \in \tangent{\bn_E^+}{\sphere{2}}
	.
\end{equation}
Here $\mylog{\bn_E^+}{\bn_E^-}$ denotes the logarithmic map, which specifies the unique tangent vector at the point $\bn_E^+$ such that the geodesic departing from $\bn_E^+$ in that direction will reach $\bn_E^-$ at unit time.
The logarithmic map is well-defined whenever $\bn_E^+ \neq -\bn_E^-$.
Moreover, $\absRiemannian{\mylog{\bn_E^+}{\bn_E^-}} = d(\bn_E^+,\bn_E^-)$ holds; see~\eqref{eq:logarithmic_map_on_the_sphere} for more details.

Together with the set of Lagrange multipliers $\bb_E\in\tangent{\bn_E^+}{\sphere{2}}$, we define the Augmented Lagrangian pertaining to \eqref{eq:geometric_inverse_problem_with_TV_of_normal_discrete} and \eqref{eq:geometric_inverse_problem_split_constraint} as
\begin{equation}
	\label{eq:Augmented_Lagrangian_rescaled_discrete}
	\LL(\Omega_h,\bd,\bb)
	\coloneqq
	\ell(u(\Omega_h),\Omega_h)
	+
	\beta \, \sum_E \absRiemannian{\bd_E} \abs{E} \\
	+
	\frac{\lambda}{2} \sum_E \abs{E} \bigabsRiemannian{\bd_E-\mylog{\bn_E^+}{\bn_E^-} - \bb_{E}}^2
	.
\end{equation}
The vectors $\bd$ and $\bb$ are simply the collections of their entries $\bd_E, \bb_E \in \tangent{\bn_E^+}{\sphere{2}}$, three components per edge $E$. 
Hence, since the tangent space $\tangent{\bn_E^+}{\sphere{2}}$ changes between shape updates, the respective quantities have to be parallely transported, which is a major difference to ADMM methods in Euclidean or Hilbert spaces.

We state the split Bregman iteration in \Cref{alg:Split_Bregman_discrete}.

\begin{algorithm}{Split Bregman method for \eqref{eq:geometric_inverse_problem_with_TV_of_normal_discrete}}
	\begin{algorithmic}[1]
		\REQUIRE Initial domain $\Omega_h^{(0)}$ 
		\ENSURE Approximate solution of \eqref{eq:geometric_inverse_problem_with_TV_of_normal_discrete}
		\STATE Set $\bb^{(0)} \coloneqq \bnull$, $\bd^{(0)} \coloneqq \bnull$
		\STATE Set $k \coloneqq 0$
		\WHILE{not converged}
		\STATE
			Perform several gradient steps for $\Omega_h \mapsto \LL(\Omega_h,\bd^{(k)},\bb^{(k)}) $ at $\Omega_h^{(k)}$ to obtain $\Omega_h^{(k+1)}$
			\label{item:alg:Split_Bregman_discrete_shape_step}
		\STATE
			Parallely transport the multiplier estimate $\bb_E^{(k)}$ on each edge $E$ from $\tangent{\bn_E^{+,(k)}}{\sphere{2}}$ to $\tangent{\bn_E^{+,(k+1)}}{\sphere{2}}$ along the geodesic from $\bn_E^{+,(k)}$ to $\bn_E^{+,(k+1)}$
			\label{item:alg:Split_Bregman_discrete_parallel_transport_step}
		\STATE
			Set $\bd^{(k+1)} \coloneqq \argmin \LL(\Omega_h^{(k+1)},\bd^{(k)},\bb^{(k)})$, see \eqref{eq:solution_of_d-problem_discrete}
			\label{item:alg:Split_Bregman_discrete_tv_step}
		\STATE 
			Update the Lagrange multipliers, i.e., set $\bb_E^{(k+1)} \coloneqq \bb_E^{(k)} + \mylog{\bn_E^{+,(k+1)}}{\bn_E^{-,(k+1)}} - \bd_E^{(k+1)}$ for all edges $E$
			\label{item:alg:Split_Bregman_discrete_multiplier_step}
		\STATE Set $k \coloneqq k+1$
		\ENDWHILE
	\end{algorithmic}
	\label{alg:Split_Bregman_discrete}
\end{algorithm}

We now address the individual steps of \cref{alg:Split_Bregman_discrete} in more detail, i.e., the successive minimization with respect to the unknown vertices of $\Omega_h$ and $\bd$, followed by an explicit update for the multiplier $\bb$.

Step~\ref{item:alg:Split_Bregman_discrete_shape_step} is the minimization of~\eqref{eq:Augmented_Lagrangian_rescaled_discrete} with respect to the unknown vertex positions of $\Omega_h$. 
To this end, we employ a gradient descent scheme, where we compute the sensitivities with respect to those node positions discretely, see \Cref{subsec:DiscreteShapeDerivative} for more details. 
Following~\cite{GoldsteinOsher2009}, an approximate minimization suffices, and thus only a certain number of steepest descent steps are performed. 
After $\Omega_h^{(k)}$ has been updated to $\Omega_h^{(k+1)}$, the quantity $\bb_E^{(k)} \in \tangent{\bn_E^{+,(k)}}{\sphere{2}}$ has to be parallely transported into the new tangent space $\tangent{\bn_E^{+,(k+1)}}{\sphere{2}}$, see step~\ref{item:alg:Split_Bregman_discrete_parallel_transport_step}, which is detailed in \eqref{eq:parallel_transport_on_the_sphere} for more details.

Step~\ref{item:alg:Split_Bregman_discrete_tv_step} is the optimization of~\eqref{eq:Augmented_Lagrangian_rescaled_discrete} with respect to $\bd$, which is a non-smoooth problem. 
It can be solved explicitly by one vectorial shrinkage operation per edge~$E$. 
Given the data $\Omega_h^{(k+1)}$ and associated normal field $\bn^{(k+1)}$, as well as multiplier $\bb_E^{(k)}$ parallely transported into $\tangent{\bn_E^{+,(k+1)}}{\sphere{2}}$, the minimizer of \eqref{eq:Augmented_Lagrangian_rescaled_discrete} is given by
\begin{equation}
	\label{eq:solution_of_d-problem_discrete}
	\alert{\bd}^{(k+1)}_E
	\coloneqq
	\max \varh\{\}{ \bigabsRiemannian{\mylog{\bn_E^{+,(k+1)}}{\bn_E^{-,(k+1)}} + \bb_E^{(k)}} - \frac{\beta}{\lambda}, \; 0 }
	\,
	\frac{\mylog{\bn_E^{+,(k+1)}}{\bn_E^{-,(k+1)}} + \bb_E^{(k)}}{\bigabsRiemannian{\mylog{\bn_E^{+,(k+1)}}{\bn_E^{-,(k+1)}} + \bb_E^{(k)}}}
\end{equation}
for each edge $E$.
Notice that \eqref{eq:solution_of_d-problem_discrete} is independent of the previous value $\bd_E^{(k)}$ and thus a parallel transport of $\bd_E^{(k)}$ into the updated tangent space is not necessary.

Step~\ref{item:alg:Split_Bregman_discrete_multiplier_step} is the multiplier update for $\bb$, which is done explicitly via
\begin{equation*}
	\bb_E^{(k+1)} = \bb_E^{(k)}+\mylog{\bn_E^{+,(k+1)}}{\bn_E^{-,(k+1)}}-\bd_E^{(k+1)}
\end{equation*}
for each edge $E$.

%------------------------------------------------------------------
\section{An EIT Model Problem and its Implementation in \fenics}
\label{sec:implementation_details}
%------------------------------------------------------------------

In this section we address some details concerning the implementation of \Cref{alg:Split_Bregman_discrete} in the finite element framework \fenics\ (version~2018.2.dev0), \cite{LoggMardalWells2012:1,AlnaesBlechtaHakeJohanssonKehletLoggRichardsonRingRognesWells2015}.  
For concreteness, we elaborate on a particular reduced loss function $\ell(u(\Omega),\Omega)$ where the state $u(\Omega)$ arises from a PDE modeling a geological electrical impedance tomography (EIT) problem with Robin-type far field boundary conditions. 
We introduce the problem under consideration first and discuss implementation details and derivative computations later on.

%------------------------------------------------------------------
\subsection{EIT Model Problem}
\label{subsec:EIT}
%------------------------------------------------------------------

\begin{figure}[htp]
	\begin{center}
		\includegraphics[width=0.46\textwidth]{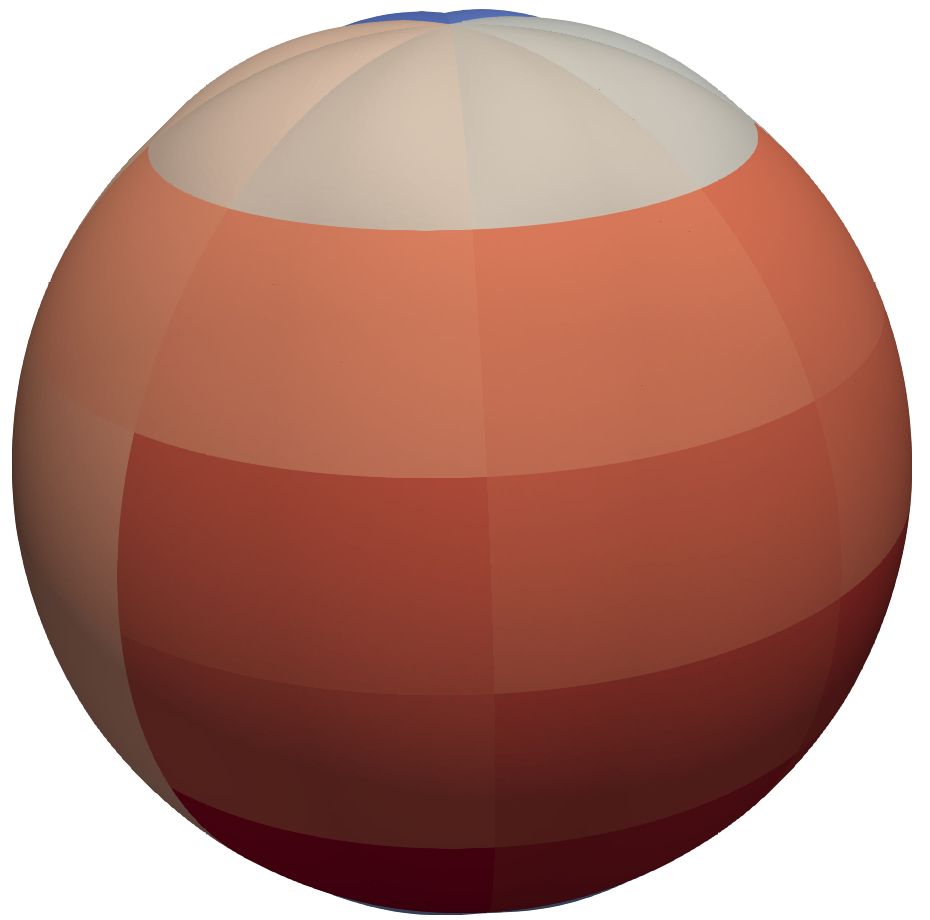}
		\hfill
		\includegraphics[width=0.49\textwidth]{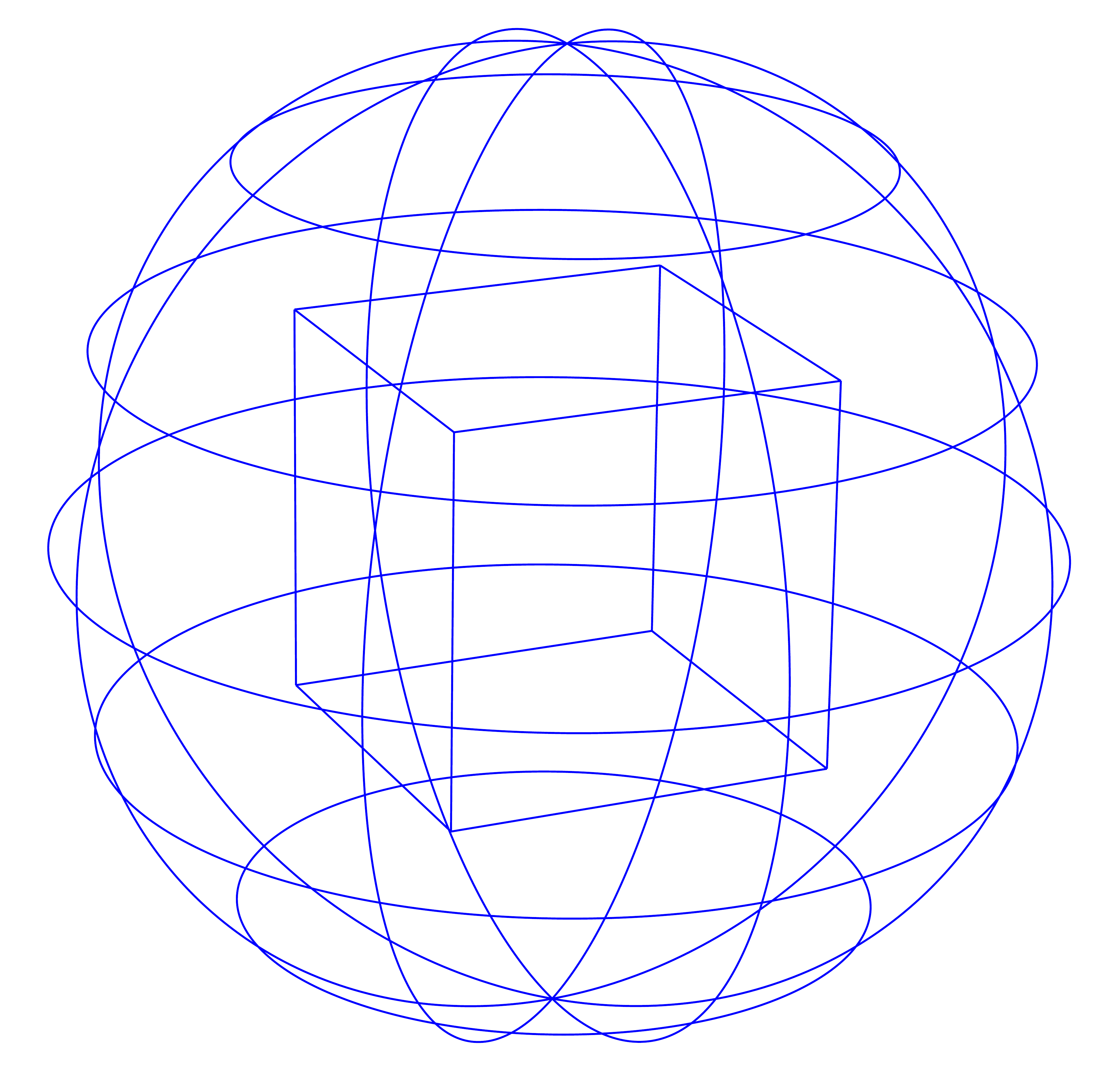}
	\end{center}
	\caption{The left plot shows the domain $\Omega$ considered in the numerical example. Each color on the outer boundary represents the support of one out of $r = 48$~electric sources $f_i$. The right figure shows a wireframe plot revealing the true inclusion $\Gamma_1$, i.e., the boundary of the cube.}
	\label{fig:EIT_Problem}
\end{figure}

Electrical impedance tomography (EIT) problems are a prototypical class of inverse problems.
Common to these problems is the task of reconstructing the internal conductivity inside a volume from boundary measurements of electric potentials or currents. 
These problems are both nonlinear and severely ill-posed and require appropriate regularization; see for instance~\cite{SantosaVogelius1990,CheneyIsaacsonNewell1999,ChungChanTai2005}. 

Traditionally, EIT problems are modeled with Neumann (current) boundary conditions and the internal conductivity is an unknown function across the entire domain.
In order to focus on the demonstration of the utility of \eqref{eq:tv_of_normal_discrete} as a regularizer in geometric inverse problems, we consider a simplified situation in which we seek to reconstruct a perfect conductor inside a domain of otherwise homogeneous electrical properties. 

Consequently, the unknowns are the vertex positions of the interface of the inclusion.
As a perfect conductor shields its interior from the electric field, there is no necessity to mesh and simulate the interior of the inclusion.
However, we mention that our methodology can be extended also to interface problems, non-perfect conductors and other geometric inverse problems.

The perfect conductor is modeled via a homogenous Neumann condition on the unknown interior boundary $\Gamma_1$ of the domain $\Omega$. 
To overcome the non-uniqueness of the electric potential, we employ Robin boundary conditions on the exterior boundary $\Gamma_2$.
The use of homogeneous Robin boundary conditions to model the far field is well-established for geological EIT problems; see, e.g., \cite{Helfrich-Schkarbanenko2011}.
We use them here also for current injection.

The geometry of our model is shown in \Cref{fig:EIT_Problem}, where $\Gamma_1$ is the unknown boundary of the perfect conductor and $\Gamma_2$ is a fixed boundary where currents are injected and measurements are taken.
We assume that $i = 1, \ldots, r \in \N$ experiments are conducted, each resulting in a measured electric potential $z_i \in \CG{1}(\Gamma_2)$, the finite element space consisting of piecewise linear, globally continuous functions on the outer boundary~$\Gamma_2$.
Experiment~\#$i$ is conducted by applying the right hand side source $f_i \in \DG{0}(\Gamma_2)$, which is the characteristic function of one of the colored regions shown in \Cref{fig:EIT_Problem}. Here, $\DG{0}$ denotes the space of piecewise constant functions.
We then seek to reconstruct the interface of the inclusion $\Gamma_1$ by solving the following regularized least-squares problem of type \eqref{eq:geometric_inverse_problem_with_TV_of_normal},
\begin{equation}
	\label{eq:MainOpt}
	\begin{aligned}
		\text{Minimize} \quad & \frac{1}{2} \sum_{i=1}^r \int_{\Gamma_2} \abs{u_i-z_i}^2 \, \ds + \beta \, \abs{\bn}_{DTV(\Gamma_1)}\\
		\text{s.t.} \quad &
		\left\{
			\begin{aligned}
				-\Delta u_i & = 0 & & \text{in } \Omega_h, \\
				\frac{\partial u_i}{\partial \bn} & = 0 & & \text{on }\Gamma_1, \\
				\frac{\partial u_i}{\partial \bn} + \alpha \, u_i & = f_i & & \text{on }\Gamma_2
			\end{aligned}
		\right.
	\end{aligned}
\end{equation}
with respect to the vertex positions of $\Gamma_1$. Here $u_i \in \CG{1}(\Omega_h)$ is the computed electric field for source $f_i$. 
Hence, the problem features $r$~PDE constraints with identical operator but different right hand sides.

As detailed in \Cref{subsec:DiscreteShapeDerivative}, we compute the shape derivative of the least-squares objective and the PDE constraint separately from the shape derivative of the regularization term. To evaluate the former, we utilize a classical adjoint approach. 
To this end, we consider the Lagrangian
\begin{multline}
	\label{eq:LagrangianSmoothPart}
	F(u_1, \ldots, u_r, p_1, \ldots, p_r, \Omega_h)
	\coloneqq 
	\\
	\sum_{i=1}^r \left[ \int_{\Gamma_2} \frac{1}{2} \abs{u_i-z_i}^2 \, \ds + \int_{\Omega_h} \nabla p_i \cdot \nabla u_i \, \d\bx + \int_{\Gamma_2} p_i (\alpha \, u_i - f_i) \, \ds \right]
\end{multline}
for $p_i \in \CG{1}(\Omega_h)$. The differentiation w.r.t.\ $u_i$ leads to the following adjoint problem for $p_i$:
\begin{equation}
	\label{eq:PDEAdjoint}
	\left\{
		\begin{aligned}
			-\Delta p_i & = 0 & & \text{in }\Omega_h, \\
			\frac{\partial p_i}{\partial \bn} & = 0 & & \text{on }\Gamma_1, \\
			\frac{\partial p_i}{\partial \bn} + \alpha \, p_i & = - (u_i - z_i) & & \text{on }\Gamma_2.
		\end{aligned}
	\right.
\end{equation}

The above adjoint PDE was implemented by hand.
Since all forward and adjoint problems are governed by the same differential operator, we assemble the associated stiffness matrix once and solve the state and adjoint equations via an ILU-preconditioned conjugate gradient method.

Provided that $u_i$ and $p_i$ solve the respective state and adjoint equations, the directional derivative of $\ell(u(\Omega_h),\Omega_h)$ coincides with the partial directional derivative of $F(u_1, \ldots, u_r, p_1, \ldots, p_r, \Omega_h)$, both with respect to the vertex positions.
In practice, we evaluate the latter using the coordinate derivative functionality of \fenics\ as described in the following subsection.

%------------------------------------------------------------------
\subsection{Discrete Shape Derivative}
\label{subsec:DiscreteShapeDerivative}
%------------------------------------------------------------------

We now focus on computing the sensitivity of finite element functionals, when mesh vertices $\bx$ of $\Omega_h$ are moved in accordance to $\bx_\varepsilon = \bx + \varepsilon \bV_{\Omega_h}$ with $\bV_{\Omega_h} \in \CG{1}^3(\Omega_h)$. As discussed in~\cite{HamMitchellPaganiniWechsung2018_preprint}, a convenient way to compute this within the finite element world is by tapping into the transformation of the reference element to the physical one. Hence, we use the symbol $\d \ell(u(\Omega_h), \Omega_h)[\bV_{\Omega_h}]$ for this object and we obtain it using the coordinate derivative functionality, first introduced in \fenics\ release 2018.2.dev0.

Our split Bregman scheme requires the shape derivative of \eqref{eq:Augmented_Lagrangian_rescaled_discrete}, which is given by
\begin{equation}
	\label{eq:LagrangianTwoTerm}
	\d\LL(\Omega_h,\bd,\bb)[P_{\Omega_h}(\bV_{\Gamma_1})] = \d\ell(u(\Omega_h), \Omega_h)[P_{\Omega_h}(\bV_{\Gamma_1})] + \d m(\Gamma_1)[\bV_{\Gamma_1}],
\end{equation}
where
\begin{equation}
	\label{eq:RegulPart}
	m(\Gamma_1) 
	\coloneqq 
	\beta \, \sum_E \absRiemannian{\bd_E} \abs{E} + \frac{\lambda}{2} \sum_E \abs{E} \bigabsRiemannian{\bd_E-\mylog{\bn_E^+}{\bn_E^-} - \bb_{E}}^2
\end{equation}
originates from the splitting approach~\eqref{eq:Augmented_Lagrangian_rescaled_discrete}.
Because our design variable is $\Gamma_1$ only, we introduce the extension $P_{\Omega_h}(\bV_{\Gamma_1})$ of $\bV_{\Gamma_1} \in \CG{1}^3(\Gamma_1)$ to the volume $\Omega_h$ by padding with zeros. Furthermore, a reduction to boundary only sensitivities can also be motivated from considering shape derivatives in the continuous setting, see~\cite[Section~3]{BergmannHerrmannHerzogSchmidtVidalNunez2019:1_preprint}.

The term $\d\ell(u(\Omega_h), \Omega_h)[P_{\Omega_h}(\bV_{\Gamma_1})]$ is computed via the adjoint approach as explained above,
\begin{equation*}
	\d\ell(u(\Omega_h), \Omega_h)[P_{\Omega_h}(\bV_{\Gamma_1})] 
	= 
	\partial_{\Omega_h} F(u_1, \ldots, u_r, p_1, \ldots, p_r, \Omega_h)[P_{\Omega_h}(\bV_{\Gamma_1})]
	.
\end{equation*}

In order to employ this AD functionality, \eqref{eq:RegulPart} needs to be given as a UFL form, a domain specific language based on Python, which forms the native language of the \fenics\ framework, see \cite{AlnaesLoggOlgaardRognesWells2014}.
Such a UFL representation is easy to achieve if all mathematical expressions are finite element functions. 
Notice that $\bd$ and $\bb$ in \eqref{eq:RegulPart} are constant functions on the edges of the boundary mesh representing $\Gamma_1$.
We can thus represent them in the so called \textit{HDivTrace} space of lowest order in \fenics.

From the directional derivatives \eqref{eq:LagrangianTwoTerm}, we pass to a shape gradient on the surface w.r.t.\ a scaled $H^1(\Gamma_1)$ scalar product by solving a variational problem.
This problem involves the weak form of a Laplace--Beltrami operator with potential term and it finds $\bW_{\Gamma_1} \in \CG{1}(\Gamma_1)^3$ such that
\begin{multline}
	\label{eq:Grad}
	\int_{\Gamma_1} 10^{-4} (\nabla \bW_{\Gamma_1}, \nabla \bV_{\Gamma_1})_2 + (\bW_{\Gamma_1}, \bV_{\Gamma_1})_2 \, \ds 
	\\
	= 
	\d\ell(u(\Omega_h),\Omega_h)[P_{\Omega_h}(\bV_{\Gamma_1})] + \d m(\Gamma_1)[\bV_{\Gamma_1}]
\end{multline}
holds for all test functions $\bV_{\Gamma_1} \in \CG{1}(\Gamma_1)^3$.

The previous procedure provides us with a shape gradient $\bW_{\Gamma_1}$ on the surface $\Gamma_1$ alone.
In order to propagate this information into the volume $\Omega_h$, we solve the following mesh deformation equation: 
find $\bW_{\Omega_h} \in \CG{1}(\Omega_h)^3$ such that
\begin{align}
	\int_{\Omega_h} (\nabla \bW_{\Omega_h}, \nabla \bV_{\Omega_h})_2 + (\bW_{\Omega_h}, \bV_{\Omega_h})_2 \, \ds = 0
\end{align}
for all test functions $\bV_{\Omega_h} \in \CG{1}(\Omega_h)^3$ with zero Dirichlet boundary conditions, where $\bW_{\Omega_h}$ is subject to the Dirichlet boundary condition $\bW_{\Omega_h} = \bW_{\Gamma_1}$ on $\Gamma_1$ and $\bW_{\Omega_h} = \bnull$ on $\Gamma_2$.
Subsequently, the vertices of the mesh are moved in the direction of $\bW_{\Omega_h}$.

%------------------------------------------------------------------
\subsection{Intrinsic Formulation Using Co-Normal Vectors}
\label{subsec:Co_normal}
%------------------------------------------------------------------

We recall that our functional of interest \eqref{eq:tv_of_normal_discrete} is formulated in terms of the unit outer normal $\bn$ of the oriented surface~$\Gamma_1$.
This leads to the term \eqref{eq:RegulPart} inside the augmented Lagrangian \eqref{eq:Augmented_Lagrangian_rescaled_discrete}.
In order to utilize the differentiation capability of \fenics\ w.r.t.\ vertex coordinates, we need to represent \eqref{eq:RegulPart} in terms of an integral. 
Since the edges are the interior facets of the surface mesh for $\Gamma_1$, and $\bd$ and $\bb$ can be represented as constant on edges as explained above, \eqref{eq:RegulPart} can indeed be written as an integral w.r.t.\ the interior facet measure \texttt{dS} on $\Gamma_1$.
Then, however, the outer normal vectors appearing in the term $\mylog{\bn_E^+}{\bn_E^-}$ are not available.
We remedy the situation by observing that the geodesic distance between two normal vectors $\bn_E^+$ and $\bn_E^-$ on the two triangles $T_1$ and $T_2$ sharing the edge $E$ can also be expressed via the co-normal (or in-plane normal) vectors $\bmu_E^+$, $\bmu_E^-$, as is shown in \Cref{fig:CoNormal}.
Indeed, one has
\begin{equation*}
	\bigabs{\mylog{\bn_E^+}{\bn_E^-}}_2 = \bigabs{\mylog{\bmu_E^+}{(-\bmu_E^-)}}_2.
\end{equation*}
Since the co-normal vectors are intrinsic to the surface $\Gamma_1$, they are available on $\Gamma_1$ while $\bn_E^+$ and $\bn_E^-$ are not.

\begin{figure}[htp]
	\centering
	\begin{overpic}[width=0.5\textwidth]{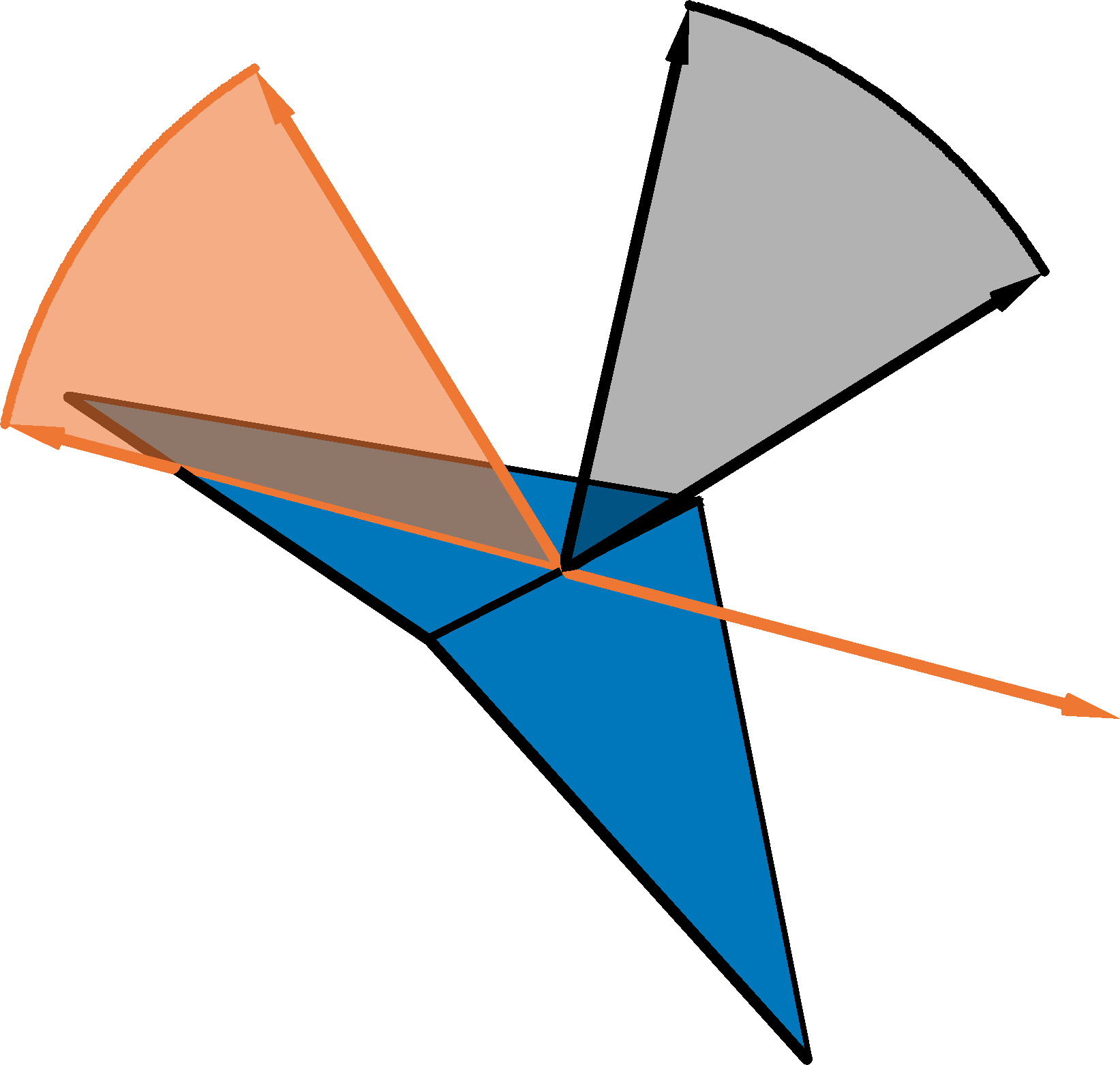}
		\put(50,35){\color{white}$T^+_E$\color{black}}
		\put(30,48){\color{white}$T^-_E$\color{black}}
		\put(80,57){$\bn_E^+$}
		\put(48,75){$\bn_E^-$}
		\put(80,28){$\bmu_E^-$}
		\put(2,48){$-\bmu_E^-$}
		\put(31,81){$\bmu_E^+$}
		\put(33,33){$E$}
	\end{overpic}
	\caption{The geodesic distance between normals $\bn_E^+$ and $\bn_E^-$ (shown in black) of two triangles $T_E^+$, $T_E^-$ which share the edge $E$ agrees with the geodesic distance between the co-normals $\bmu_E^+$ and $-\bmu_E^-$ (shown in orange).}
	\label{fig:CoNormal}
\end{figure}
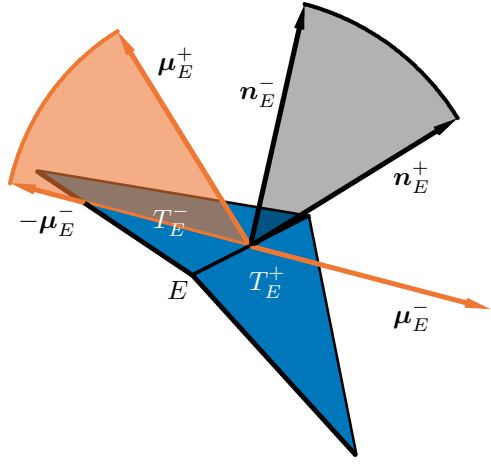

%------------------------------------------------------------------
\section{Numerical Results}
\label{sec:numerical_results}
%------------------------------------------------------------------

In this section we present numerical results obtained with \Cref{alg:Split_Bregman_discrete} for the geological impedance tomography model problem described in the previous section.
The data of the problem are given in \Cref{tab:parameters_of_model_problem} and the initial guess of the inclusion $\Gamma_1$, as well as the true inclusion, are shown in  \Cref{fig:initial_and_target}.
The state $u$ and adjoint state $p$ were discretized using piecewise linear, globally continuous finite elements on a tetrahedral grid of $\Omega$ minus the volume enclosed by $\Gamma_1$.
The mesh has $\num{4429}$~vertices and $\num{19384}$~tetrahedra. Regarding the shape optimization problem of \Cref{alg:Split_Bregman_discrete}, we perform $10$ gradient steps per split Bregman iteration combined with an Armijo linesearch with starting step size of $10^2$. Also, we stop the whole algorithm, i.e., the outer Bregman iteration, when the initial gradient of the above mentioned shape optimization problem has a norm below $10^{-7}$ in the sense of \eqref{eq:Grad}. 
\begin{table}
	\centering
	\begin{tabular}{ll}
		\toprule
		domain $\Omega$ & unit sphere $B_1(0) \setminus [-0.4,0.4]^3 \subset \R^3$ \\
		measurement boundary $\Gamma_2$ & boundary of $\Omega$ \\
		boundary of true inclusion $\Gamma_1$ & boundary of $[-0.4,0.4]^3$ \\
		initial guess for $\Gamma_1$ & boundary of $B_{0.5}(0) \subset \R^3$ \\
		number of measurements & $r = 48$ \\
		Robin coefficient in \eqref{eq:MainOpt} & $\alpha = 10^{-5}$ \\
		split Bregman parameter & $\lambda = 10^{-5}$ \\
		standard deviation of noise & $\sigma = 0$ or $\sigma = 0.34\cdot10^{-2}$ \\
		regularization parameter \ldots & \\ 
		\quad for total variation regularization & $\beta = 10^{-6}$ \\
		\quad for surface area regularization & $\gamma = 5\cdot10^{-5}$, $2\cdot10^{-5}$ \\
		shape step size & $10^2$\\
		\bottomrule
		\\
	\end{tabular}
	\caption{Setting of the numerical experiments for \eqref{eq:MainOpt}.}
	\label{tab:parameters_of_model_problem}
\end{table}

In \Cref{fig:results_IOP}, we show the results obtained in the noise-free setting (top row) and with noise (bottom row).
In the latter case, normally distributed random noise is added with zero mean and standard deviation $\sigma = 0.34\cdot 10^{-2}$ per degree of freedom of $z_i$ on $\Gamma_2$ for each of the $r = 48$~simulations of the forward model \eqref{eq:MainOpt}. The amount of noise is considerable when put in relation to the average range of values for the simulated states, which is
\begin{equation*}
	\frac{\sum_{i=1}^{r}\big(\max_{\bs \in \Gamma_2} z_i(\bs) - \min_{\bs \in \Gamma_2} z_i(\bs)\big)}{r} \approx 0.34, \quad i = 1,\ldots, r.
\end{equation*}
Due to mesh corruption, we have to remesh $\Omega_h$ at some point in the cases with noise. Afterwards, we start again \Cref{alg:Split_Bregman_discrete} with the remeshed $\Omega_h$ as new initial guess.

For comparison, we also provide results obtained for a related problem in \Cref{fig:results_IOP}, using the popular surface area regularization with the same data otherwise. For the surface area regularization, $\beta \, \abs{\bn}_{TV(\Gamma_1)}$ is replaced by $\gamma \, \int_{\Gamma_1} \ds = \gamma \, \sum_F \abs{F}$, where $F$ are the facets of $\Gamma_1$.
Because the problem is smooth in this case, we apply a shape gradient scheme directly rather than a split Bregman scheme and terminate as soon as the norm of the gradient falls below $5\cdot 10^{-8}$.
The regularization parameters $\beta$ and $\gamma$ are selected by hand in each case.
Automatic parameter selection strategies can clearly be applied here as well, but this is out of the scope of the present paper.

\begin{figure}[htp]
	\centering
	\includegraphics[width=0.48\textwidth]{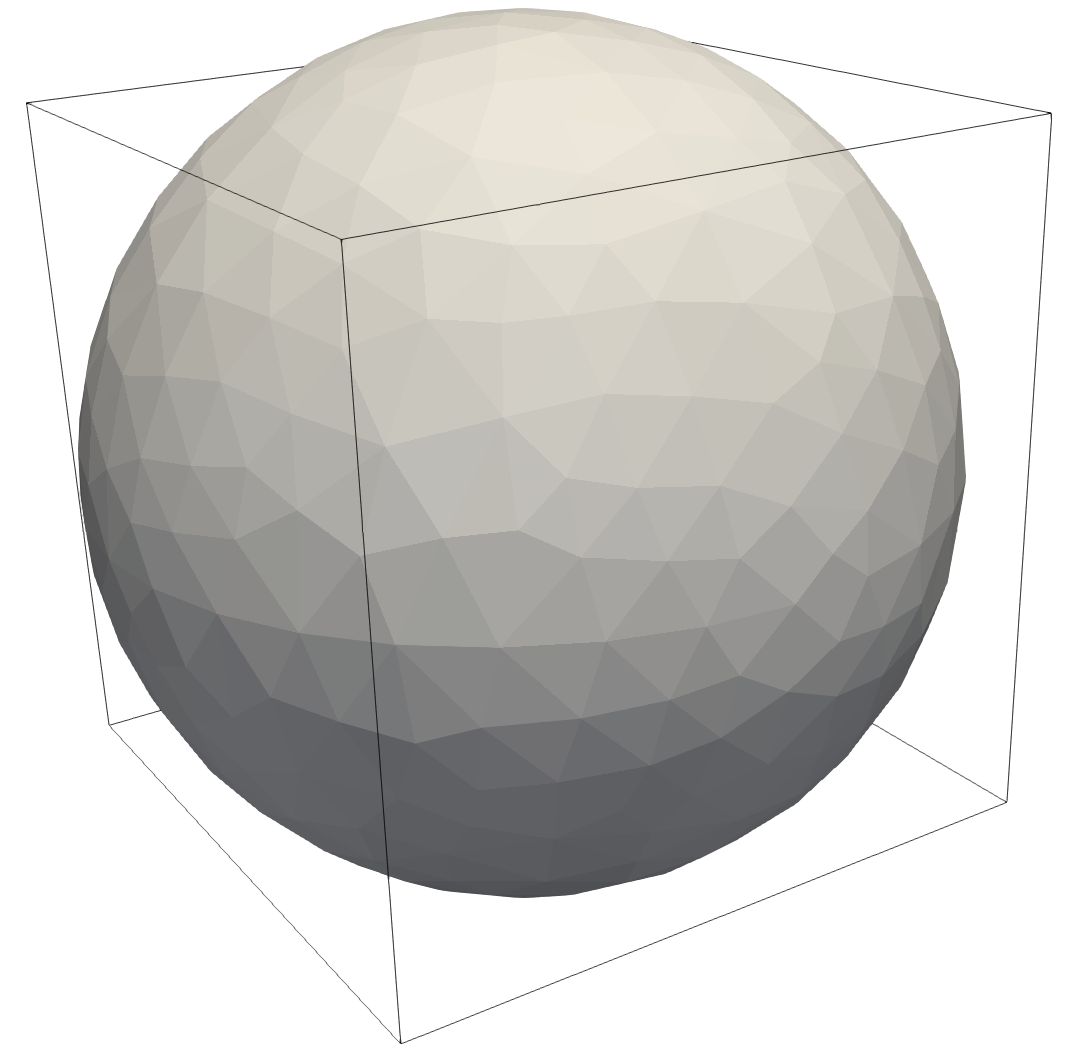}
	\includegraphics[width=0.48\textwidth]{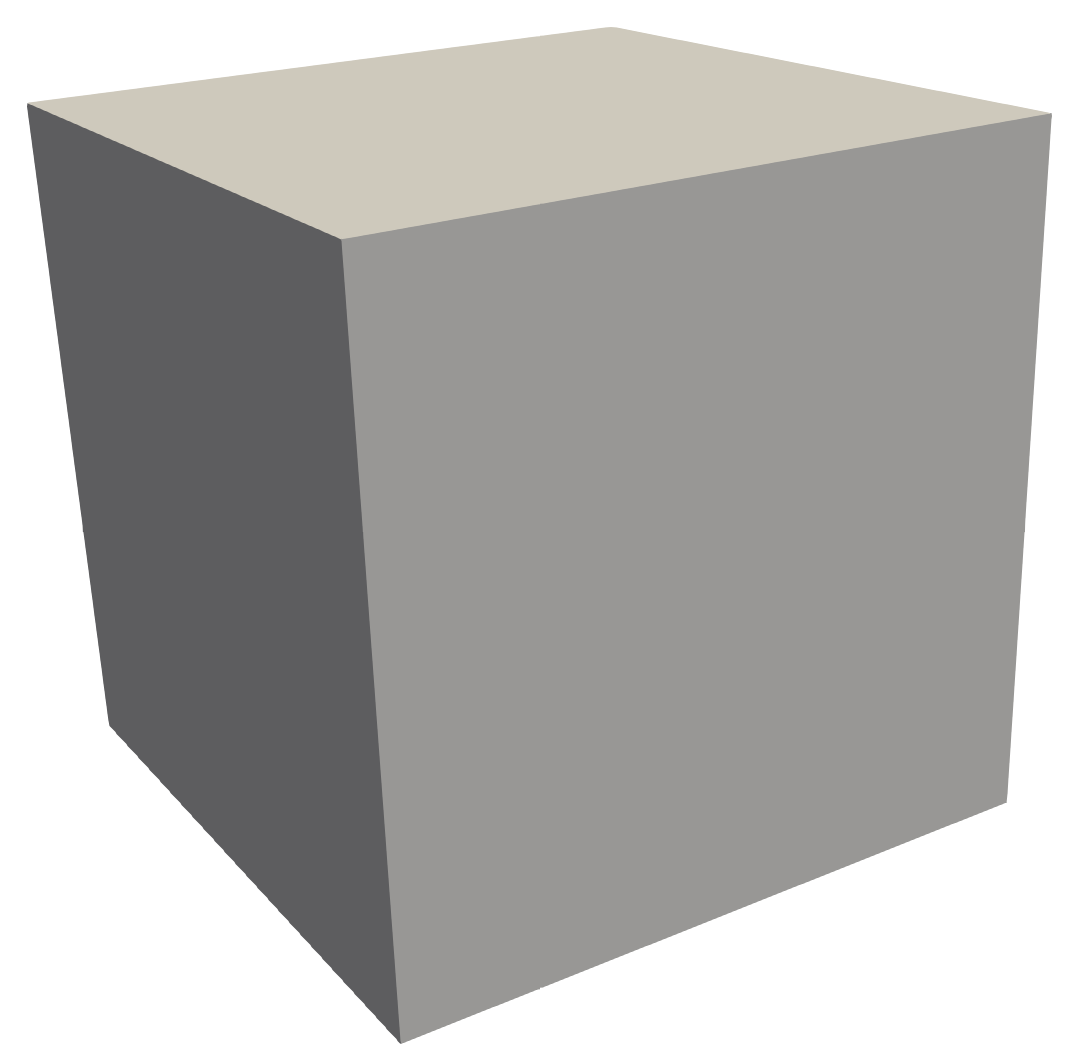}
	\caption{Initial guess for the inclusion $\Gamma_1$ on the left and the true inclusion on the right.}
	\label{fig:initial_and_target}
\end{figure}

\begin{figure}[htp]
	\centering
	\includegraphics[width=0.32\textwidth]{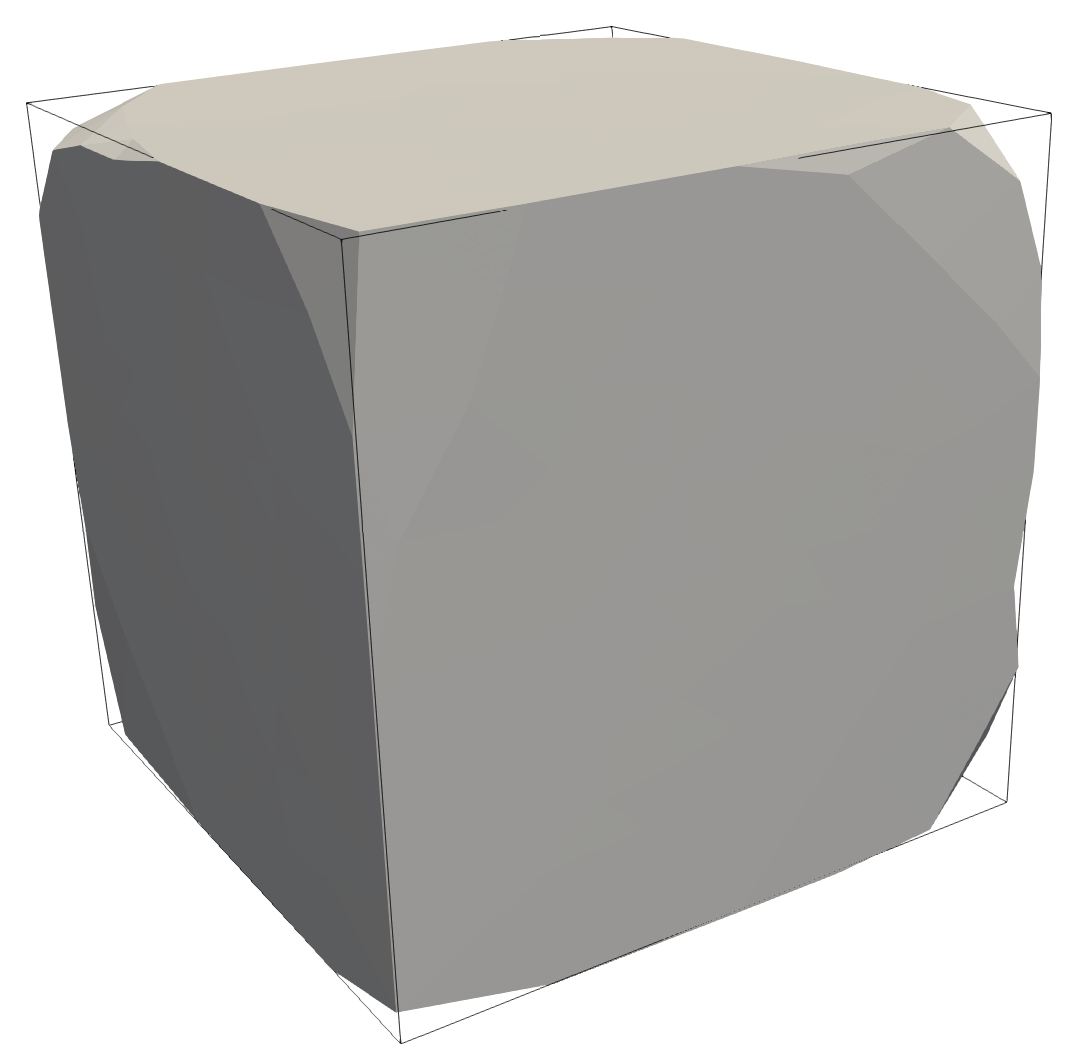}
	\includegraphics[width=0.32\textwidth]{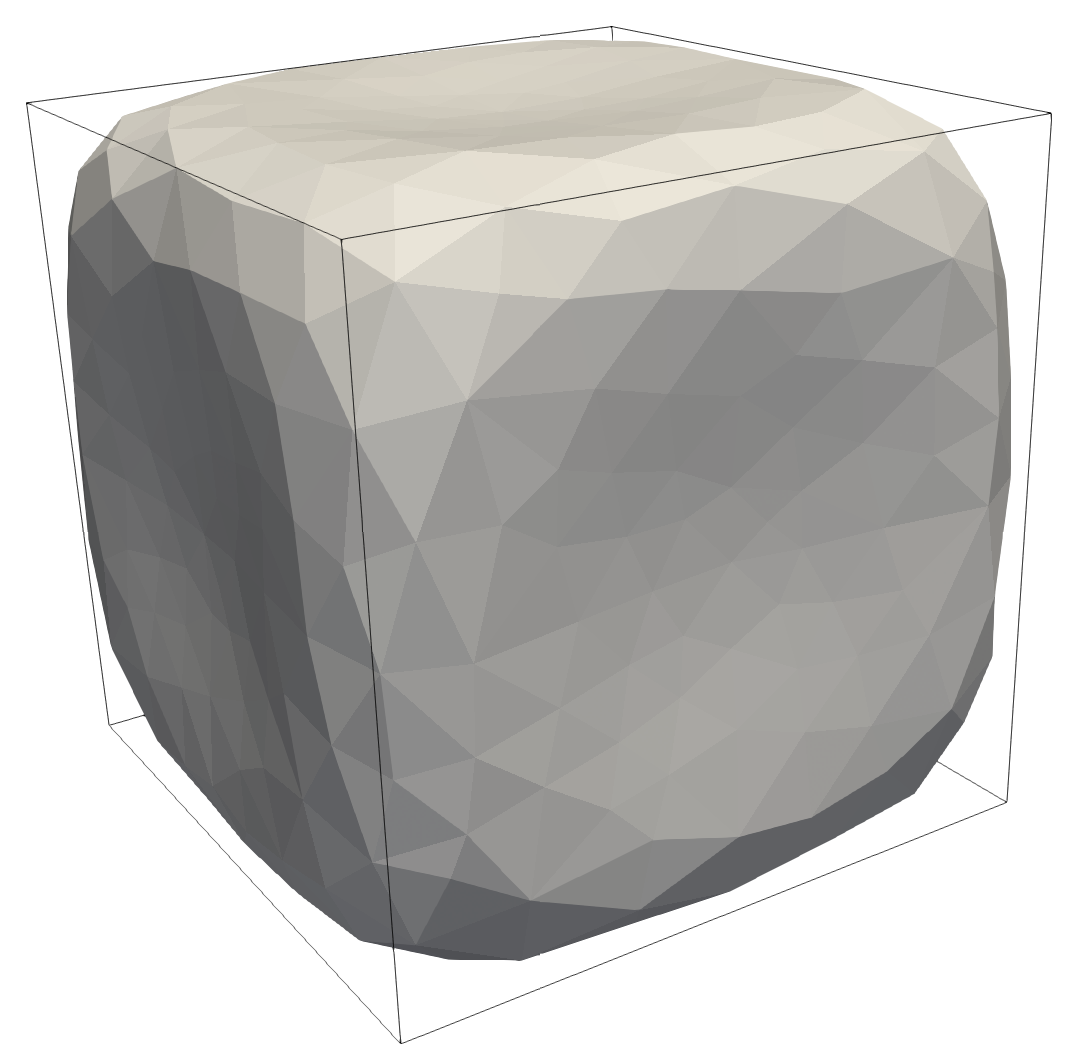}
	\includegraphics[width=0.32\textwidth]{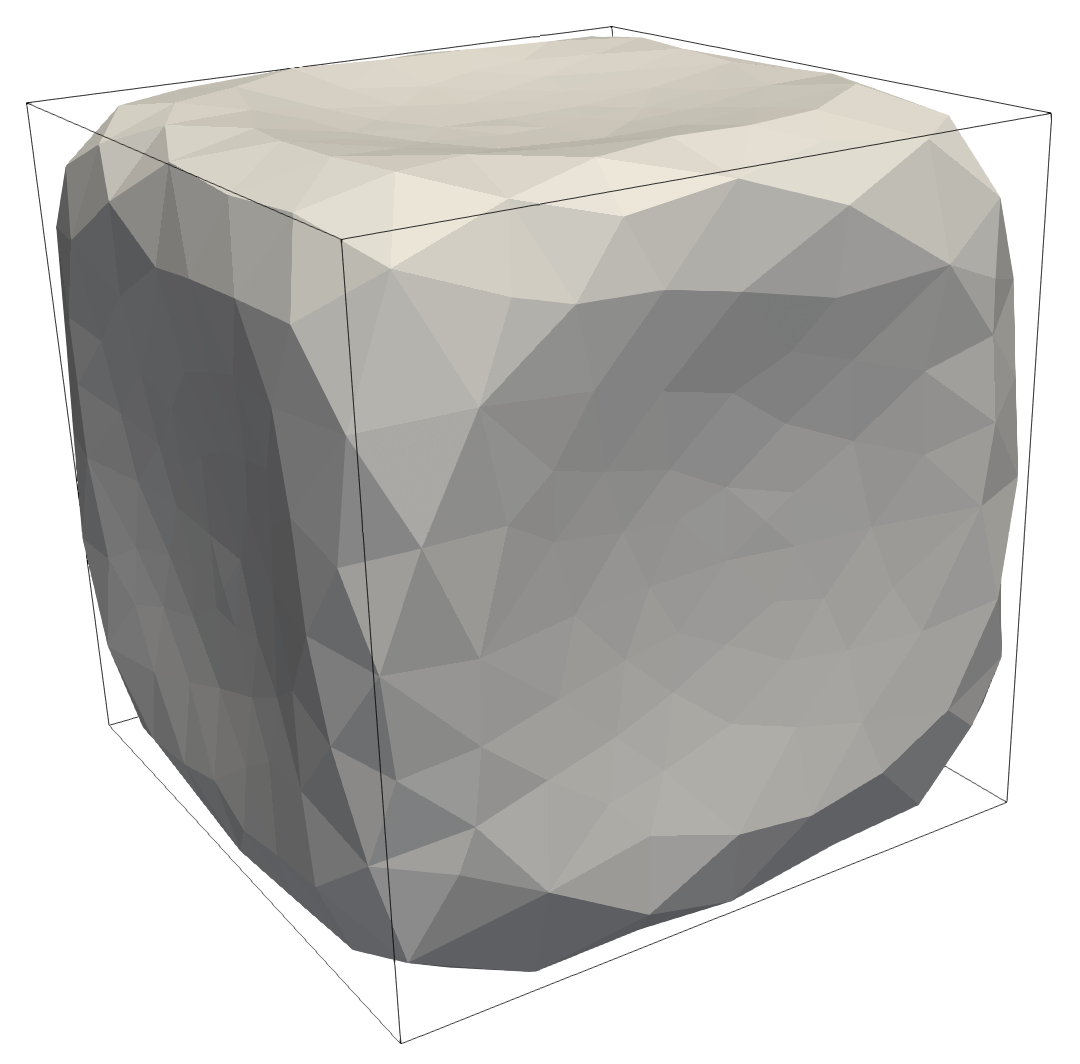}\\
	\includegraphics[width=0.32\textwidth]{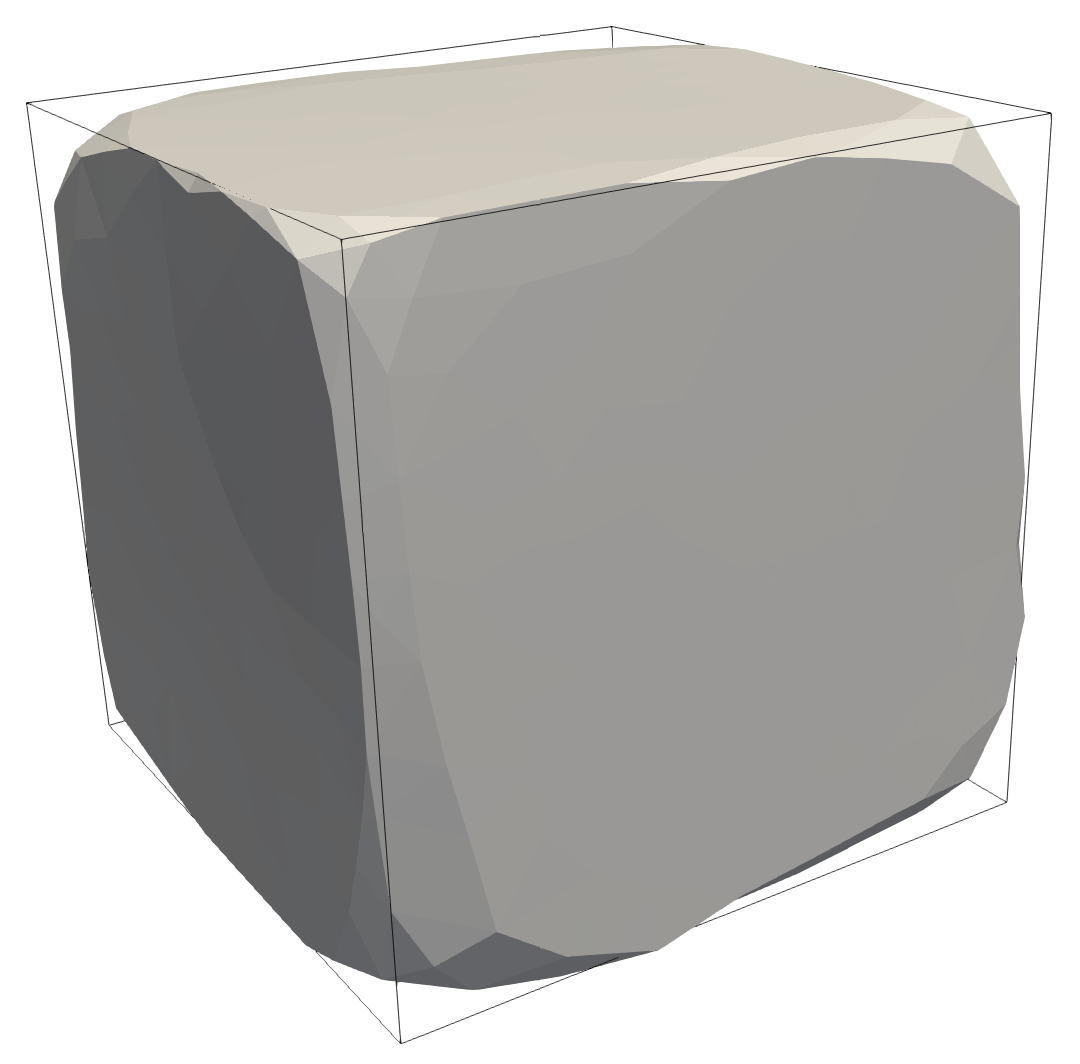}
	\includegraphics[width=0.32\textwidth]{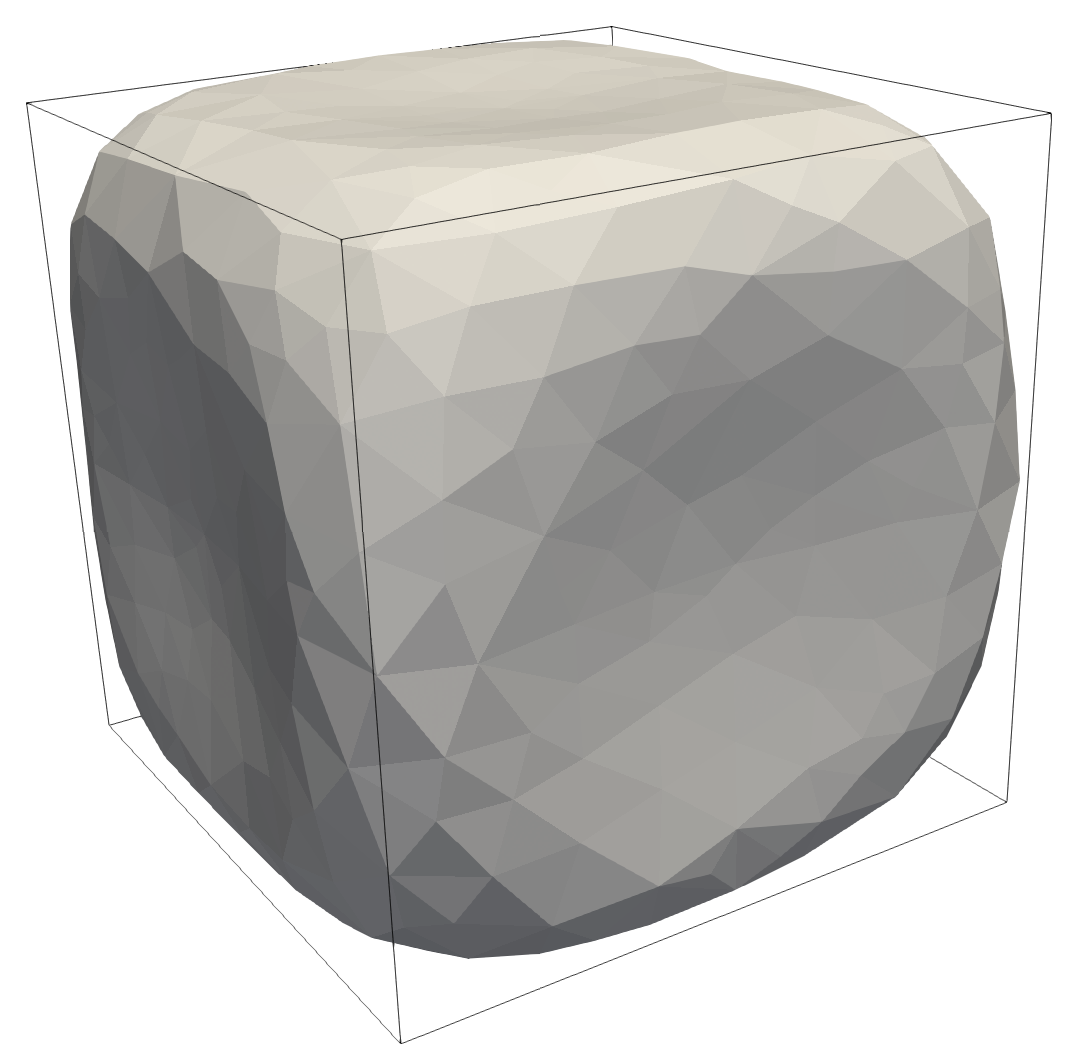}
	\includegraphics[width=0.32\textwidth]{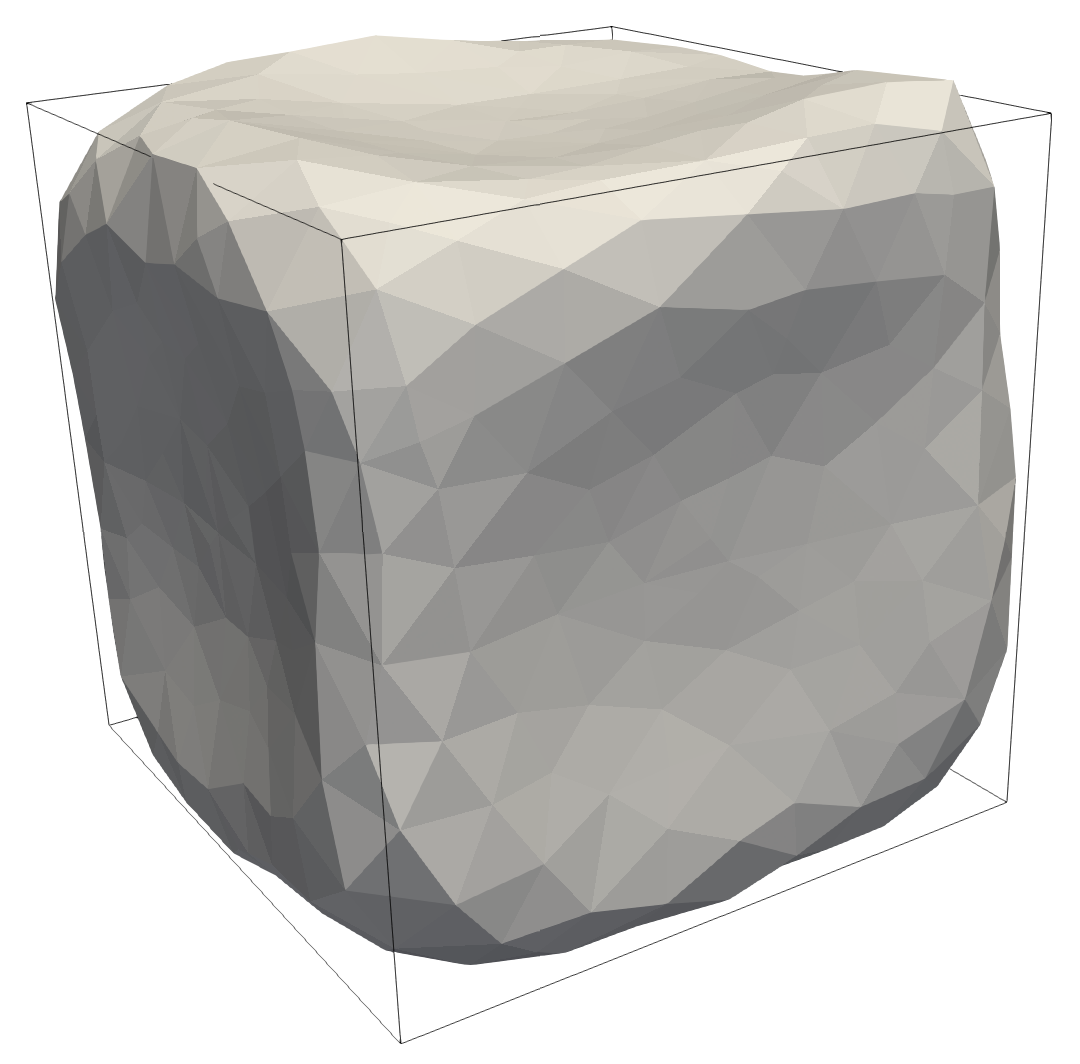}
	\caption{Top row: setting without noise; left: total variation regularization, $\beta = 10^{-6}$ and $90$~iterations; middle: surface area regularization with $\gamma = 5\cdot 10^{-5}$ and $1129$~iterations; right: surface area regularization with $\gamma = 2\cdot 10^{-5}$ and $978$~iterations. Bottom row: setting with noise; left: total variation regularization with $\beta = 10^{-6}$ and $173$~iterations with remeshing after iteration~$121$; middle: surface area regularization with $\gamma = 5\cdot 10^{-5}$ and $1016$~iterations with remeshing after iteration~$539$; right: surface area regularization with $\gamma = 2\cdot 10^{-5}$ and $987$~iterations with remeshing after iteration~$308$.}
	\label{fig:results_IOP}
\end{figure}
As is expected and well known, the use of surface area regularization leads to results in which the identified inclusion $\Gamma_1$ is smoothed out.
This can be explained by the observation that the gradient based minimization of the surface area yields a mean curvature flow.
By contrast, our novel prior \eqref{eq:tv_of_normal_discrete_repeated} allows for piecewise flat shapes and thus the interface $\Gamma_1$ is closely reconstructed in the noise-free situation.
Even in the presence of noise, the reconstruction is remarkably good.
In particular, the flat lateral surfaces and sharp edges can be identified quite well.

%------------------------------------------------------------------
\section{Conclusions}
\label{sec:conclusions}
%------------------------------------------------------------------

In this paper we introduced a discrete analogue of the total variation prior for the normal vector field as shown in \cite{BergmannHerrmannHerzogSchmidtVidalNunez2019:1_preprint}.
While we are currently unable to characterize all minimizers of its discrete counterpart, we showed that the icosahedron and a cube with crossed diagonals are stationary under an area constraint.
We conjecture that the full set of minimizers is much richer than this, in particular when the connectivity is included as design unknown.
It has been argued in \cite[Section~3.3]{PellisKilianDellingerWallnerPottmann2019} that minimal energy is achieved for meshes which are not triangular, but whose faces are approximately rectangular.

We proposed, described and implemented a split Bregman (ADMM) scheme for the numerical solution of shape optimization problems involving the discrete total variation of the normal.
In contrast to a Euclidean ADMM as proposed for instance in \cite{GoldsteinOsher2009}, the normal vector data belongs to the sphere $\sphere{2}$.
Therefore, the formulation of the ADMM method requires concepts from differential geometry.
In particular, the discrete setting utilizes logarithmic maps and parallel transport of tangent vectors.
An analysis of the ADMM scheme is beyond the scope of this paper and will be presented elsewhere.

We demonstrate the utility of the discrete total variation of the normal as a shape prior in a geometric inverse problem, in which we aim to detect a polyhedral inclusion.
Unlike the popular surface area regularization, our prior allows for piecewise flat shapes.

\appendix
%------------------------------------------------------------------
\section{The Sphere as a Riemannian Manifold}
\label{sec:sphere_as_a_Riemannian_manifold}
%------------------------------------------------------------------

In this section we provide some useful formulas for the sphere
\begin{equation*}
	\sphere{2} = \{ \bn \in \R^3: \abs{\bn}_2 = 1 \}
\end{equation*}
equipped with the Riemannian metric obtained from the pull back of the Euclidean metric from the ambient space $\R^3$.
We are going to represent points $\bn \in \sphere{2}$ by vectors in $\R^3$.
Moreover, we identify the tangent space at $\bn$ with the two-dimensional subspace
\begin{equation*}
	\tangent{\bn}{\sphere{2}}
	=
	\{ \bxi \in \R^3: \bxi^\top \bn = 0 \}
	.
\end{equation*}
We utilize the Riemannian metric $\Riemannian{\ba}{\bb} = \ba^\top \bb$ in $\tangent{\bn}{\sphere{2}}$ and the norm $\absRiemannian{\ba} = (\ba^\top \ba)^{1/2}$.

The geodesic distance between any two $\bn, \bn' \in \sphere{2}$ is given by
\begin{equation}
	\label{eq:geodesic_distance_on_the_sphere}
	d(\bn,\bn') 
	= 
	\arccos(\bn^\top \bn')
	.
\end{equation}

The geodesic curve $\geodesic{\bn}{\bxi}{\,\cdot\,} \colon \R \to \sphere{2}$ departing from $\bn \in \sphere{2}$ in the direction of $\bxi \in \tangent{\bn}{\sphere{2}}$ is given by
\begin{equation}
	\label{eq:geodesic_on_the_sphere}
	\geodesic{\bn}{\bxi}{t} 
	=
	\cos \bigh(){t \, \absRiemannian{\bxi}}\bn + \sin \bigh(){t \, \absRiemannian{\bxi}}\frac{\bxi}{\absRiemannian{\bxi}}.
\end{equation}
The exponential map is thus given by
\begin{equation}
	\label{eq:exponential_map_on_the_sphere}
	\myexp{\bn}{\bxi}
	=
	\geodesic{\bn}{\bxi}{1} 
	=
	\cos \bigh(){\absRiemannian{\bxi}} \, \bn + \sin \bigh(){\absRiemannian{\bxi}}\frac{\bxi}{\absRiemannian{\bxi}}
	.
\end{equation}
The logarithmic map is the inverse of the exponential map w.r.t.\ to the tangent direction $\bxi$.
In other words, $\bxi = \mylog{\bn}{\bn'}$ holds if any only if $\bxi$ is the unique element in $\tangent{\bn}{\sphere{2}}$ such that $\myexp{\bn}{\bxi} = \bn'$ holds.
The logarithmic map is well-defined whenever $\bn \neq -\bn'$ holds.
In this case, we have
\begin{equation}
	\label{eq:logarithmic_map_on_the_sphere}
	\mylog{\bn}{\bn'} 
	= 
	d(\bn,\bn') \frac{\bn'-(\bn^\top \bn') \, \bn}{\absRiemannian{\bn'-(\bn^\top \bn') \, \bn}}
	.
\end{equation}
Finally we require the concept of parallel transport of a tangent vector from one tangent space to another, along the unique shortest geodesic connecting the base points.
Specifically, the parallel transport $P_{\bn\to \bn'} \colon \tangent{\bn}{\sphere{2}} \to \tangent{\bn'}{\sphere{2}}$ along the unique shortest geodesic $\geodesic{\bn}{\mylog{\bn}{\bn'}}{\,\cdot\,}$ connecting $\bn$ and $\bn' \neq -\bn$ is given by
\begin{equation}
	\label{eq:parallel_transport_on_the_sphere}
	\begin{aligned}
		P_{\bn\to \bn'}(\bxi)
		&
		=
		\bxi - \frac{\bxi^\top(\mylog{\bn}{\bn'})}{d^2(\bn,\bn')}(\mylog{\bn}{\bn'} + \mylog{\bn'}{\bn})
		\\
		&
		=
		\bxi + \bigh(){\cos(\absRiemannian{\bv}) \, \bu - \bu - \sin(\absRiemannian{\bv}) \, \bn} \, \bu^\top \bxi
		,
	\end{aligned}
\end{equation}
see for instance \cite{HosseiniUschmajew2017} and \cite[Section~2.3.1]{Persch2018}, repectively.
Here we used the abbreviations $\bv = \mylog{\bn}{\bn'}$, $\absRiemannian{\bv} = d(\bn,\bn')$ and $\bu = \frac{\bv}{\absRiemannian{\bv}}$.
To see that both expressions in~\eqref{eq:parallel_transport_on_the_sphere} coincide ---after plugging in
the definition of the geodesic distance~\eqref{eq:geodesic_distance_on_the_sphere}--- it remains to show that
\begin{equation*}
	-\frac{\bn^\top\bn'\mylog{\bn}{\bn'}}{\absRiemannian{\mylog{\bn}{\bn'}}}
	+ \sqrt{1-(\bn^\top\bn')^2} \, \bn
	= 
	\frac{\mylog{\bn'}{\bn}}{\absRiemannian{\mylog{\bn'}{\bn}}}.
\end{equation*}
which holds true since the norm of the logarithmic map is
\begin{equation*}
	\absRiemannian{\mylog{\bn}{\bn'}} 
	= 
	\absRiemannian{\bn' - \bn^\top\bn'\bn} 
	= 
	\sqrt{ (\bn'^\top\bn') - (\bn^\top\bn')}
	=
	\sqrt{ 1 - (\bn^\top\bn')}
	= \absRiemannian{\mylog{\bn'}{\bn}}.
\end{equation*}
Hence multiplying with the denominator of the first term in \eqref{eq:parallel_transport_on_the_sphere} yields the equality with the second term, since using the definition of the logarithmic map we obtain
\begin{equation*}
	(\bn^\top\bn') \, \bn - (\bn^\top\bn')^2 \bn - (1-(\bn^\top\bn')^2) \, \bn 
	= 
	\bn - (\bn^\top\bn') \, \bn'.
\end{equation*}

%------------------------------------------------------------------
\subsection*{Acknowledgments}
%------------------------------------------------------------------

The authors would like to thank two anonymous reviewers for their constructive criticism which helped improve the paper.

This work was supported by DFG grants HE~6077/10--1 and SCHM~3248/2--1 within the \href{https://spp1962.wias-berlin.de}{Priority Program SPP~1962} (\emph{Non-smooth and Complementarity-based Distributed Parameter Systems: Simulation and Hierarchical Optimization}), which is gratefully acknowledged.

% Print the bibliography
\ifcsname newrefcontext\endcsname
	\newrefcontext[sorting=nyt]
\printbibliography
\else
%\printbibliography[sorting=nyt]
\fi

\end{document}